\documentclass[11pt]{article}

\usepackage[T1]{fontenc}
\usepackage[english]{babel}
\usepackage[utf8]{inputenc}
\usepackage{amsmath,amssymb,amsfonts,amsthm}
\usepackage{graphicx}
\usepackage{enumitem}
\usepackage{bbold}
\usepackage{tikz}
\usepackage{caption}
\usepackage{subcaption}
\usepackage{comment}

\usepackage{comment}
\usepackage{listings} 
\usepackage{circuitikz}
\usepackage{authblk}
\usepackage[scale=0.75]{geometry}
\usepackage{stmaryrd} 
\usepackage[normalem]{ulem}

\theoremstyle{plain}
\newtheorem{theo}{Theorem}
\newtheorem{lemma}[theo]{Lemma}
\newtheorem{proposition}[theo]{Proposition}

\newtheorem{conjecture}[theo]{Conjecture}

\theoremstyle{definition}
\theoremstyle{remark}

\usepackage{color}
\usepackage{xcolor}
\usepgflibrary{patterns}
\usepgflibrary[patterns]
\usetikzlibrary{patterns}
\usetikzlibrary[patterns]

\newcommand{\ra}{\hspace{.1cm}\rightarrow\hspace{.1cm}} 
\newcommand{\la}{\hspace{.1cm}\leftarrow\hspace{.1cm}} 
\newcommand{\ua}{\hspace{.1cm}\uparrow\hspace{.1cm}} 
\newcommand{\da}{\hspace{.1cm}\downarrow\hspace{.1cm}} 
\newcommand{\nea}{\hspace{.1cm}\nearrow\hspace{.1cm}} 
\newcommand{\nwa}{\hspace{.1cm}\nwarrow\hspace{.1cm}} 
\newcommand{\sea}{\hspace{.1cm}\searrow\hspace{.1cm}} 
\newcommand{\swa}{\hspace{.1cm}\swarrow\hspace{.1cm}} 
\newcommand{\Stairs}{\mathcal{S}} 
\newcommand{\Diag}{\mathcal{D}} 
\newcommand{\name}{complete neighbor }
\newcommand{\eps}{\varepsilon}

\newcommand{\NN}{\mathbb{N}}

\newcommand{\RR}{\mathbb{R}}
\newcommand{\ZZ}{\mathbb{Z}}
\newcommand{\C}{\mathcal{C}}

\newcommand{\nn}{\texttt{NN}}
\newcommand{\cn}{\texttt{CN}}
\newcommand{\dec}{\texttt{dec}}

\newcommand{\step}{\mathcal{E}}
\newcommand{\gras}{\textbf}

\definecolor{codegreen}{rgb}{0,0.6,0}
\definecolor{codegray}{rgb}{0.5,0.5,0.5}
\definecolor{codepurple}{rgb}{0.58,0,0.82}
\definecolor{backcolour}{rgb}{0.94,0.94,0.94}
\definecolor{airforceblue}{rgb}{0.36, 0.54, 0.66}
\definecolor{beaublue}{rgb}{0.74, 0.83, 0.9}
\definecolor{coolblack}{rgb}{0.0, 0.18, 0.39}
\definecolor{aliceblue}{rgb}{0.94, 0.97, 1.0}
\definecolor{darkpastelgreen}{rgb}{0.01, 0.75, 0.24}
\lstdefinestyle{mystyle}{
    backgroundcolor=\color{backcolour},   
    commentstyle=\color{codegreen},
    keywordstyle=\color{magenta},
    numberstyle=\tiny\color{codegray},
    stringstyle=\color{codepurple},
    basicstyle=\ttfamily\footnotesize,
    breakatwhitespace=false,         
    breaklines=true,                 
    captionpos=b,                    
    keepspaces=true,                 
    numbers=left,                    
    numbersep=5pt,                  
    showspaces=false,                
    showstringspaces=false,
    showtabs=false,                  
    tabsize=2
}

\title{Gilbert's disc model conditioned on the square lattice}
\author[1]{J\'er\^ome Casse}
\author[2]{Ir\`ene Marcovici}
\author[2]{Maxence Poutrel}
\affil[1]{\small Université Paris-Saclay, CNRS,
Laboratoire de mathématiques d’Orsay, 91405 Orsay, France\\ \texttt{jerome.casse@universite-paris-saclay.fr}}
\affil[2]{\small Univ Rouen Normandie, CNRS, Normandie Univ, LMRS UMR 6085, F-76000 Rouen, France\newline \texttt{\{irene.marcovici,maxence.poutrel1\}@univ-rouen.fr}}

\begin{document}

\maketitle

\begin{abstract}
    We present a new percolation model on the two-dimensional lattice, which can be seen as a conditioned version of continuous percolation on the plane. Let us place a point uniformly at random in each cell of the grid $\ZZ^2$. These points correspond to the vertices of our graph, and we connect two points by an edge if their distance is less than a fixed radius $R$. We are interested in the radius from which there exists almost surely an infinite connected component. We also study two other critical radii specific to the geometry of our model: the smallest radius such that there exists a positioning of the points for which there is an infinite connected component, and the radius from which all points are connected to each other.
\end{abstract}

\section{Introduction}\label{sec:Intro}

Let $R\in \RR_+$ be a fixed radius, and let $d:\RR^2\times\RR^2\to\RR_+$ be a distance on $\RR^2$. In the following of the article, we mainly focus on $\mathcal{L}_p$ distance.
We consider a collection $P=(P_{i,j})_{(i,j)\in\ZZ^2} = (X_{i,j},Y_{i,j})_{(i,j)\in\ZZ^2}$ of points of $\RR^2$, such that for any $(i,j)\in\ZZ^2$, $X_{i,j}\in [i,i+1]$ and $Y_{i,j}\in [j,j+1]$. We denote the set of such collections by $\mathcal{P}=\prod_{(i,j)\in\ZZ^2} [i,i+1]\times[j,j+1]$.

\paragraph{Complete neighbor model.} For any $P\in\mathcal{P}$ we construct the undirected graph $\mathcal{G}(P) = (P,E)$ where $E = \{(P_u,P_v) : d(P_u,P_v)<R\}$, see the graph on the left of Figure~\ref{fig:Sim+Rep}.

We focus mainly on the projected version of $\mathcal{G}(P)$ on the lattice $\ZZ^2$, see the graph on the right of Figure~\ref{fig:Sim+Rep}, which we denote by $\Gamma_{d,R}^{\cn}(P)=(\ZZ^2,E_{d,R}^{\cn}(P))$ defined by 
$$E_{d,R}^{\cn}(P)=\{(u,v)\in\ZZ^2\times\ZZ^2: u \neq v \text{ and } d(P_u,P_v)\leq R\}.$$
We say that such a graph $\Gamma_{d,R}^{\cn}(P)$ is a configuration  of parameter $R$ for the \emph{complete neighbor model}. 
For a graph $G=(V,E)$, we say that two vertices $u, v\in V$ are \emph{neighbors} if $(u,v)\in E$. A \emph{path} between $u$ and $v$ is a sequence $(u_k)_{0\leq k\leq n}$ of vertices, with $u_0=u$ and $u_n=v$, such that $\forall k\in\{0,...,n-1\}$, $u_k$ and  $u_{k+1}$ are neighbors. We say that $u$ and $v$ are \emph{connected} if there exists a path between $u$ and $v$. When $V=\ZZ^2$, we denote by $\C_0(G)$ the \emph{connected component} of $(0,0)$, that is the set of vertices connected to $(0,0)$.

Our main focus is the typical connected components of $\Gamma_{d,R}^{\cn}(P)$ when the points of $P$ are chosen independently at random, with $P_{i,j}$ uniformly distributed in $[i,i+1]\times[j,j+1]$ for each $(i,j)\in \ZZ^2$. We denote by $\mathbb{P}$ the corresponding point distribution on $\mathcal{P}$, that is, the product distribution $\mathbb{P}=\bigotimes_{(i,j)\in\ZZ^2}\mathcal{U}([i,i+1]\times [j,j+1])$, where $\mathcal{U}(A)$ denotes the uniform distribution on $A$. 
Figure~\ref{fig:Sim+Rep} presents a simulation of the lattice-based Gilbert's disc model for the Euclidean distance. \par

\paragraph{Context and connection to Gilbert's disc model.} The complete neighbor model was first introduced in \cite{HM90}, as a new type of neighborhood for cellular automata, which allows one to observe more regular patterns than in the classical setting. 

Observe that it can also be seen as a conditioned version of Gilbert's disc model, for which points are distributed in the plane according to a Poisson process. In the most general framework of the Boolean model, also known as continuous percolation, the radii $R$ associated with the points form a family of i.i.d.~random variables~\cite{Gil61,MR96}. 

On its side, the nearest neighbour model that we introduce a little later can be interpreted as a variant of independent Bernoulli percolation on the square lattice~\cite{Gri99}, that presents local dependencies. 

Whether for discrete or continuous percolation models, a fundamental question that has been the subject of extensive research concerns the existence or not of an infinite connected component, depending on the values of the parameters.

\paragraph{Percolation radii.}
In our context, we define the \emph{percolation probability} as the probability that the connected component $\mathcal{C}_0(\Gamma_{d,R}^{\cn}(P))$ is infinite, that is, 
$$\theta_d^{\cn}(R) = \mathbb{P}(|\C_0(\Gamma_{d,R}^{\cn}(P))|=\infty),$$
and the \emph{critical radius} by
$$R_c^{\cn}(d) = \sup\{R>0 : \theta_d^{\cn}(R) = 0\}.$$

\begin{proposition}\label{prop:coupling}
    Let $d_1$ and $d_2$ be two distances on $\RR^2$, and let $R_1,R_2>0$.
    If for any $z\in\RR^2$, we have $B_{d_1}\left(z,R_1\right)\subset B_{d_2}\left(z,R_2\right)$,
    then for any $P\in\mathcal P$, $E_{d_1,R_1}^{\cn}(P)\subset E_{d_2,R_2}^{\cn}(P)$, and as a consequence, $\theta^{\cn}_{d_1}(R_1)\leq\theta^{\cn}_{d_2}(R_2).$
\end{proposition}
 
The proof is straightforward. 
As a consequence of Proposition~\ref{prop:coupling}, the percolation probability $R\mapsto\theta_d^{\cn}(R)$ is non-decreasing. Moreover, if the model is invariant by translation (that is the case later when we consider a distance derived from a norm), there exists two distinct regimes: for $R<R_c^{\cn}(d)$, the graph $\Gamma_{d,R}^{\cn}(P)$ has almost surely no infinite connected component (sub-critical regime), while for $R>R_c^{\cn}(d)$, it has almost surely at least one infinite connected component (super-critical regime).

The geometry of our model also leads us to introduce two other types of critical radius, namely the total connectivity radius and the possible connectivity radius. 

The \emph{total connectivity radius} is the smallest radius from which all points are inside an infinite connected component. It is defined by
$$R_{\max}^{\cn}(d)=\inf\{R>0 : \forall P\in \mathcal{P}, |\mathcal{C}_0(\Gamma_{d,R}^{\cn}(P))|=\infty\}.$$
The \emph{possible connectivity radius} is the smallest radius from which an infinite connected component becomes possible, in the sense that 
$$R_{\min}^{\cn}(d) = \inf\{R>0 : \exists P\in\mathcal{P}, |\C_0(\Gamma_{d,R}^{\cn}(P))|=\infty\}.$$

We clearly have the following inequalities
\begin{equation*}
    R_{\min}^{\cn}(d) \leq R_c^{\cn}(d) \leq R_{\max}^{\cn}(d).
\end{equation*}

\paragraph{Nearest neighbor model.} For a given distance $d$, we also consider the \emph{nearest neighbor} model that consists in restricting the possible connections to the four adjacent cells.
Precisely, the set of edges $E_{d,R}^{\nn}$ of the new graph $\Gamma_{d,R}^{\nn}(P)$ obtained is then given by
$$E_{d,R}^{\nn}=E_{d,R}^{\cn} \cap \{(u,v)\in\ZZ^2\times \ZZ^2 : ||u-v||_1=1\}.$$

\begin{figure}
    \centering
    \hspace{-1.5cm}\begin{minipage}{.35\textwidth}
    \scalebox{.65}{%
    \input{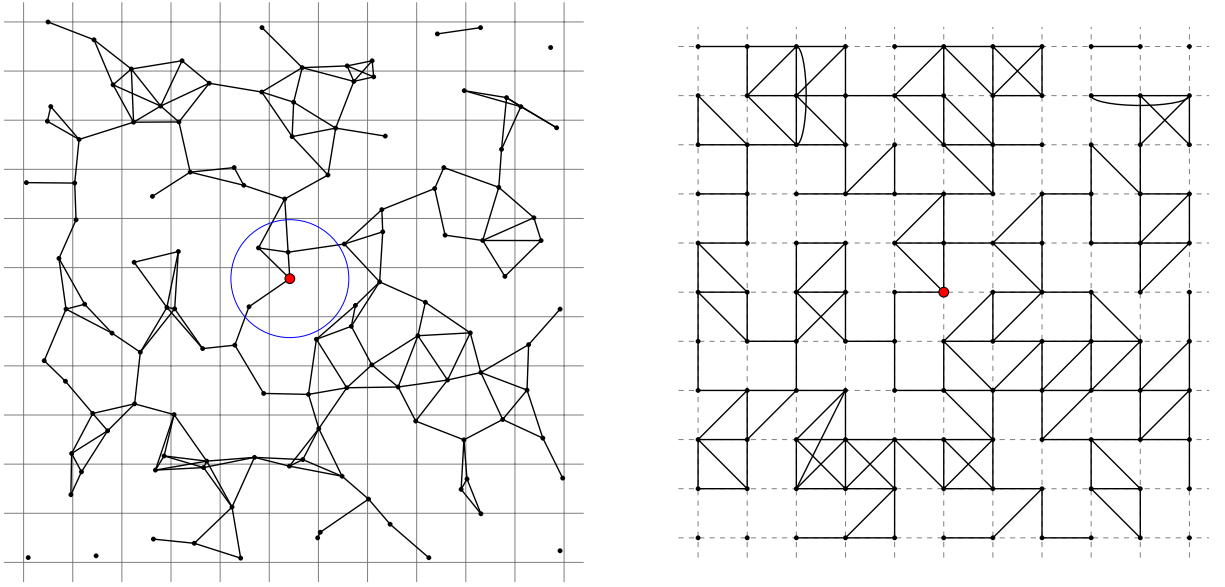}
    }
    \end{minipage}
    \hspace{3cm}
    \begin{minipage}{.35\textwidth}
    \scalebox{.65}{%
    \begin{tikzpicture}[scale = 1]
[inner sep=2mm,    place/.style={circle,draw=blue!50,fill=blue!20,thick},    transition/.style={rectangle,draw=black!50,fill=black!20,thick}]
\draw[step=1cm,color=gray,ultra thin,dashed] (-5.4,-5.4) grid (5.4,5.4);
\begin{scope}[thick]
\draw (-5,-4) -- (-5,-3);
\draw (-5,-4) -- (-4,-4);
\draw (-5,-3) -- (-4,-3);
\draw (-5,-3) -- (-4,-2);
\draw (-5,-2) -- (-5,-1);
\draw (-5,-2) -- (-4,-2);
\draw (-5,-1) -- (-5,0);
\draw (-5,0) -- (-5,1);
\draw (-5,0) -- (-4,0);
\draw (-5,1) -- (-4,1);
\draw (-5,2) -- (-4,2);
\draw (-5,3) -- (-5,4);
\draw (-5,3) -- (-4,3);
\draw (-5,5) -- (-4,5);
\draw (-4,-4) -- (-5,-3);
\draw (-4,-4) -- (-4,-3);
\draw (-4,-3) -- (-4,-2);
\draw (-4,-3) -- (-3,-2);
\draw (-4,-2) -- (-3,-2);
\draw (-4,-1) -- (-5,0);
\draw (-4,-1) -- (-4,0);
\draw (-4,-1) -- (-3,-1);
\draw (-4,0) -- (-5,1);
\draw (-4,1) -- (-4,2);
\draw (-4,2) -- (-4,3);
\draw (-4,3) -- (-5,4);
\draw (-4,3) -- (-3,3);
\draw (-4,4) -- (-4,5);
\draw (-4,4) -- (-3,4);
\draw (-4,4) -- (-3,5);
\draw (-4,5) -- (-3,5);
\draw (-3,-5) -- (-2,-5);
\draw (-3,-4) -- (-3,-3);
\draw (-3,-4) -- (-2,-4);
\draw (-3,-4) -- (-2,-3);
\draw (-3,-4) -- (-2,-2);
\draw (-3,-3) -- (-2,-3);
\draw (-3,-3) -- (-2,-2);
\draw (-3,-2) -- (-3,-1);
\draw (-3,-2) -- (-2,-2);
\draw (-3,-1) -- (-3,0);
\draw (-3,-1) -- (-2,0);
\draw (-3,0) -- (-3,1);
\draw (-3,0) -- (-2,0);
\draw (-3,0) -- (-2,1);
\draw (-3,1) -- (-2,1);
\draw (-3,2) -- (-2,2);
\draw (-3,3) -- (-4,4);
\draw (-3,3) -- (-3,4);
\draw (-3,3) arc [start angle=-90, end angle=90,x radius=.2cm, y radius=1.0cm];
\draw (-3,3) -- (-2,3);
\draw (-3,4) -- (-3,5);
\draw (-3,4) -- (-2,4);
\draw (-3,4) -- (-2,5);
\draw (-3,5) -- (-2,5);
\draw (-2,-5) -- (-1,-5);
\draw (-2,-5) -- (-1,-4);
\draw (-2,-4) -- (-3,-3);
\draw (-2,-4) -- (-2,-3);
\draw (-2,-4) -- (-1,-4);
\draw (-2,-4) -- (-1,-3);
\draw (-2,-3) -- (-2,-2);
\draw (-2,-3) -- (-1,-3);
\draw (-2,-1) -- (-3,0);
\draw (-2,-1) -- (-2,0);
\draw (-2,-1) -- (-1,-1);
\draw (-2,0) -- (-2,1);
\draw (-2,2) -- (-2,3);
\draw (-2,2) -- (-1,2);
\draw (-2,2) -- (-1,3);
\draw (-2,3) -- (-3,4);
\draw (-2,3) -- (-2,4);
\draw (-2,4) -- (-2,5);
\draw (-2,4) -- (-1,4);
\draw (-1,-5) -- (-1,-4);
\draw (-1,-4) -- (-2,-3);
\draw (-1,-4) -- (-1,-3);
\draw (-1,-3) -- (0,-3);
\draw (-1,-2) -- (-1,-1);
\draw (-1,-2) -- (0,-2);
\draw (-1,-1) -- (-1,0);
\draw (-1,0) -- (0,0);
\draw (-1,1) -- (0,1);
\draw (-1,1) -- (0,2);
\draw (-1,2) -- (-1,3);
\draw (-1,2) -- (0,2);
\draw (-1,4) -- (0,4);
\draw (-1,4) -- (0,5);
\draw (-1,5) -- (0,5);
\draw (0,-5) -- (1,-5);
\draw (0,-4) -- (-1,-3);
\draw (0,-4) -- (0,-3);
\draw (0,-4) -- (1,-4);
\draw (0,-4) -- (1,-3);
\draw (0,-3) -- (1,-3);
\draw (0,-2) -- (0,-1);
\draw (0,-2) -- (1,-2);
\draw (0,-1) -- (1,-1);
\draw (0,-1) -- (1,0);
\draw (0,0) -- (-1,1);
\draw (0,0) -- (0,1);
\draw (0,1) -- (0,2);
\draw (0,1) -- (1,1);
\draw (0,2) -- (1,2);
\draw (0,3) -- (-1,4);
\draw (0,3) -- (0,4);
\draw (0,3) -- (1,3);
\draw (0,4) -- (0,5);
\draw (0,5) -- (1,5);
\draw (1,-5) -- (2,-4);
\draw (1,-4) -- (0,-3);
\draw (1,-4) -- (1,-3);
\draw (1,-4) -- (2,-4);
\draw (1,-3) -- (0,-2);
\draw (1,-3) -- (1,-2);
\draw (1,-2) -- (0,-1);
\draw (1,-2) -- (2,-2);
\draw (1,-2) -- (2,-1);
\draw (1,-1) -- (1,0);
\draw (1,-1) -- (2,-1);
\draw (1,-1) -- (2,0);
\draw (1,0) -- (2,0);
\draw (1,1) -- (2,1);
\draw (1,1) -- (2,2);
\draw (1,2) -- (0,3);
\draw (1,2) -- (1,3);
\draw (1,3) -- (0,4);
\draw (1,3) -- (1,4);
\draw (1,3) -- (2,3);
\draw (1,4) -- (0,5);
\draw (1,4) -- (1,5);
\draw (1,4) -- (2,4);
\draw (1,4) -- (2,5);
\draw (1,5) -- (2,5);
\draw (2,-5) -- (2,-4);
\draw (2,-5) -- (3,-5);
\draw (2,-3) -- (2,-2);
\draw (2,-3) -- (3,-3);
\draw (2,-3) -- (3,-2);
\draw (2,-2) -- (2,-1);
\draw (2,-2) -- (3,-2);
\draw (2,-2) -- (3,-1);
\draw (2,-1) -- (3,-1);
\draw (2,0) -- (1,1);
\draw (2,0) -- (2,1);
\draw (2,0) -- (3,0);
\draw (2,1) -- (2,2);
\draw (2,2) -- (3,2);
\draw (2,4) -- (1,5);
\draw (2,4) -- (2,5);
\draw (3,-4) -- (3,-3);
\draw (3,-4) -- (4,-4);
\draw (3,-3) -- (4,-3);
\draw (3,-2) -- (3,-1);
\draw (3,-2) -- (4,-2);
\draw (3,-2) -- (4,-1);
\draw (3,-1) -- (3,0);
\draw (3,-1) -- (4,-1);
\draw (3,1) -- (3,2);
\draw (3,1) -- (4,1);
\draw (3,2) -- (3,3);
\draw (3,4) -- (4,4);
\draw (3,4) arc [start angle=180, end angle=360,x radius=1.0cm, y radius=.2cm];
\draw (3,5) -- (4,5);
\draw (4,-5) -- (3,-4);
\draw (4,-5) -- (4,-4);
\draw (4,-4) -- (3,-3);
\draw (4,-3) -- (4,-2);
\draw (4,-3) -- (5,-3);
\draw (4,-3) -- (5,-2);
\draw (4,-2) -- (4,-1);
\draw (4,-2) -- (5,-2);
\draw (4,-2) -- (5,-1);
\draw (4,-1) -- (3,0);
\draw (4,0) -- (4,1);
\draw (4,0) -- (5,1);
\draw (4,1) -- (4,2);
\draw (4,1) -- (5,1);
\draw (4,1) -- (5,2);
\draw (4,2) -- (3,3);
\draw (4,2) -- (4,3);
\draw (4,2) -- (5,2);
\draw (4,3) -- (4,4);
\draw (4,3) -- (5,4);
\draw (4,4) -- (5,4);
\draw (5,-4) -- (5,-3);
\draw (5,-3) -- (5,-2);
\draw (5,-2) -- (5,-1);
\draw (5,-1) -- (5,0);
\draw (5,1) -- (5,2);
\draw (5,3) -- (4,4);
\draw (5,3) -- (5,4);
\end{scope}
\foreach \i in {-5,...,5}{ 
    \foreach \j in {-5,...,5}{
        \node [fill=black,inner sep=1pt,shape=circle] at (\i,\j) {};
    }
}
\node [fill=red,inner sep=2pt,shape=circle,draw=black] at (0,0) {};
\end{tikzpicture}
    }
    \end{minipage}
    \caption{On the left, a simulation of the lattice-based Gilbert's disc model for the Euclidean distance (parameter $p=2$), with $R=1.2$, and on the right, the graph $\Gamma_{d,R}^{\cn}(P)$ resulting from this simulation.}
    \label{fig:Sim+Rep}
\end{figure}

 We denote the corresponding radii with a $\nn$ exponent. Again, we have
\begin{equation*}
    R_{\min}^{\nn}(d) \leq R_c^{\nn}(d) \leq R_{\max}^{\nn}(d).
\end{equation*}
And since $E_{d,R}^{\nn}\subset E_{d,R}^{\cn}$, it follows that

\begin{displaymath}
    R_{\min}^{\cn}(d) \leq R_{\min}^{\nn}(d)\text{, }
    R_{c}^{\cn}(d) \leq R_{c}^{\nn}(d)\text{, and }
    R_{\max}^{\cn}(d) \leq R_{\max}^{\nn}(d).
\end{displaymath}

\begin{proposition}\label{prop:couplingNN}
    Proposition~\ref{prop:coupling} still holds for the nearest neighbor model.
\end{proposition}

\paragraph{$\boldmath \mathcal{L}^p-$norm.}
 Now, we consider only distances $d$ resulting from a $\mathcal{L}^p-$norm on $\RR^2$, i.e.\ for any $Q,Q' \in \RR^2$, $d(Q,Q')=\|Q-Q'\|_p$, where for $Q=(x,y)$,

$$\|Q\|_p = \left \{
    \begin{array}{ll}
		\big(|x|^p + |y|^p)^{\frac{1}{p}} &\mbox{ if } p\in[1,\infty);\\
		\max\{|x|,|y|\} &\mbox{ if } p=\infty.\\
	\end{array}
	\right.
$$

From now on, when the distance considered results from the $\mathcal{L}^p-$norm, we put the value $p\in[1,\infty]$ as a parameter of the graph, the set of edges and the radii instead of $d$. 

We now state the main theorems of this article, which provide exact values or boundaries of the different radii when the distance $d$ is derived from $\mathcal{L}^p-$norm, for both the \name model and the nearest neighbor one. 

\begin{proposition}
    Let $p_1,p_2\in[1,\infty]$ with $p_1\leq p_2$. We have
    \begin{displaymath}
    R_{\min}^{\cn}(p_2) \leq R_{\min}^{\cn}(p_1)\text{, }
    R_{c}^{\cn}(p_2) \leq R_{c}^{\cn}(p_1)\text{, and }
    R_{\max}^{\cn}(p_2) \leq R_{\max}^{\cn}(p_1).
\end{displaymath}
    The same property holds for the nearest neighbor model.
\end{proposition}

\begin{proof}
    By Proposition~\ref{prop:coupling}, we have for any $P$ and $R$, $E_{p_1,R}^{\cn}(P)\subset E_{p_2,R}^{\cn}(P)$. Especially, if the connected component is infinite for the distance derived from the $\mathcal{L}^{p_1}-$norm, it is also infinite when the distance considered is derived from the $\mathcal{L}^{p_2}-$norm.
    This conclude the proof for the complete neighbor model. 
    By Proposition~\ref{prop:couplingNN}, the arguments are similar for the nearest neighbor model.
\end{proof}

\paragraph{Main results.} For the total connectivity radius, the exact values are given in the following theorem.
\begin{theo}\label{thm:Rmax}
    For $p\in[1,\infty]$, we have 
    $$R_{\max}^{\cn}(p) = R_{\max}^{\nn}(p) = \|(2,1)\|_p.$$
\end{theo}

The proof is given in Section~\ref{sec:Rmax}.\par
\smallskip
For the possible connectivity radius, we also give the exact values. It is totally explicit for the \name model, but "slightly less" for the nearest neighbor variation.
\begin{theo}\label{thm:Rmin}
    \begin{enumerate}
        \item For $p\in[1,\infty)$, $R_{\min}^{\cn}(p) = \displaystyle \min \left\{\frac{1}{2},\frac{2}{1+2^{2-\frac{1}{p}}}\right\}$
        and $R_{\min}^{\cn}(\infty)=2/5$.
        \item For $p\in[1,\infty)$, we denote by $R(p)$ the unique solution on $[0,1]$ of the equation
            \begin{equation}\label{eqRmin4V}
                (2-R)^p + (3-3R)^p = (4R)^p.
            \end{equation}
            We have $R_{\min}^{\nn}(p) \displaystyle = \min\left\{\frac{1}{2},R(p)\right\}$ and $R_{\min}^{\nn}(\infty) = 3/7$.
    \end{enumerate}
\end{theo}

An important part of this article is devoted to the proof of this result, which is given in Section~\ref{sec:Rmin}.\par

\smallskip

Finally, we improve the bounds for the critical radius.

\begin{theo}\label{thm:Rc}
    For $p\in[1,\infty]$, we have 
    \begin{enumerate}
        \item $R_{\min}(p)\leq R_{c}^{\cn}(p)\leq\displaystyle \left\|\left(1,\frac{3}{2}\right)\right\|_p $,
        \item $\sqrt{2-\sqrt{2}}\leq R_{c}^{\nn}(p)\leq\displaystyle \left\|\left(1,\frac{3}{2}\right)\right\|_p $.
    \end{enumerate}
\end{theo}

The proof is given in Section~\ref{sec:Rc}.

The results of those three theorems are summarized on Figure~\ref{fig:ResultsCN} for the \name model and on Figure~\ref{fig:ResultsNN} for the nearest neighbor one.

\begin{figure}
    \centering
    \begin{tikzpicture}
	\begin{scope}[x=3cm,y=2cm,xshift=-3cm]
	\fill[red!10] [smooth, samples = 200, domain={1}:{4}]
    	(1,3.5) -- plot (\x,{(2^(\x) + 1)^(1/(\x))}) -- (4,3.5) -|  cycle;
	\fill[green!10] [smooth, samples = 200, domain={1}:{4}]
    	(1,1) -- plot (\x,{(2^(\x) + 1)^(1/(\x))}) -- (4,1) -|  cycle;
	\fill[white] [smooth, samples = 200, domain={1}:{4}]
    	(1,1) -- plot (\x,{((3/2)^(\x) + 1)^(1/(\x))}) -- (4,1) -|  cycle;
	\fill[blue!10] [smooth, samples = 200, domain={2.4904}:{4}]
    	(1,0) -- (1,1/2) -- (2.4094,1/2) -- plot (\x,{2/(1+2^(2-1/\x))}) -- (4,0) -|  cycle;
	\node [draw=none,fill=none] at (2.5,3) {\color{red}$|\C_0(\Gamma_{d,R}^{\cn}(P))| = \infty$};
	\node [draw=none,fill=none] at (2.5,1.95) {\color{green}$\mathbb{P}(|\C_0(\Gamma_{d,R}^{\cn}(P))| = \infty)>0$};
	\node [draw=none,fill=none] at (2.5,.25) {\color{blue}$|\C_0(\Gamma_{d,R}^{\cn}(P))| < \infty$};
	\end{scope}
	\begin{scope}[y=2cm]
	\draw [->] (0,0) -- (0,3.5);
	\node [anchor=south east,fill=white,draw=none] at (-.1,3.5) {$R$};
	\node [anchor= east,fill=none,draw=none] at (-.1,0) {0};
	\node [anchor=east,fill=none,draw=none] at (0,1/2) {$\frac{1}{2}$};
	\node [anchor=east,fill=white,draw=none] at (-.1,1) {$1$};
	\node [anchor=east,fill=white,draw=none] at (-.1,2) {$2$};
	\node [anchor=east,fill=white,draw=none] at (-.1,3) {$3$};
	\draw (-0.1,0) -- (0,0);
	\draw (-0.1,1/2) -- (0,1/2);
	\draw (-0.1,1) -- (0,1);
	\draw (-0.1,2) -- (0,2);
	\draw (-0.1,3) -- (0,3);
	\end{scope}
	
	\begin{scope}[x=3cm,xshift=-3cm]
		\draw [->] (1,0) -- (4,0);
		 \node [anchor=north west,fill=white,draw=none] at (4,0) {$p$};
		 \node [anchor=north,fill=white,draw=none] at (1,-.1) {$1$};
		 \node [anchor=north,fill=white,draw=none] at (2,-.1) {$2$};
		 \node [anchor=north,fill=white,draw=none] at (3,-.1) {$3$};
		 \node [anchor=north,fill=white,draw=none] at (2.4094,-.1) {\tiny$\frac{\ln(2)}{\ln(\frac{4}{3})}$};
		 \draw (1,-0.1) -- (1,0);
		 \draw (2,-0.1) -- (2,0);
		 \draw (3,-0.1) -- (3,0);
		 \draw (2.4094,-0.1) -- (2.4094,0);
	\end{scope}
	
	\begin{scope}[x=3cm, y=2cm, xshift=-3cm]
	\draw [color=blue] (1,1/2) -- (2.4094,1/2);
	\node [fill=blue,inner sep=2pt,shape=circle] at (2.4094,1/2) {};
	\draw[color=blue,domain=2.4094:4] plot (\x,{2/(1+2^(2-1/\x))}) node[right,fill=white,draw=none] {$R_{\min}^{\cn}\to\frac{2}{5}$};
	
	 \draw[color=red,domain=1:4] plot (\x,{(2^(\x) + 1)^(1/(\x))}) node[right,fill=white,draw=none] {$R_{\max}^{\cn}\to2$};
	 
	\draw[color=gray,dashed,domain=1:4] plot (\x,{((3/2)^(\x) + 1)^(1/(\x))}) node[right,fill=none,draw=none] {{\tiny Upper Bound} $R_c^{\cn}$};
	 
	\node [inner sep=2pt,fill=none,draw=none] at (1,1.305) {\tiny x};
	\node [inner sep=2pt,fill=none,draw=none] at (1.5,1.12) {\tiny x};
	\node [inner sep=2pt,fill=none,draw=none] at (2,1.04) {\tiny x};
	\node [inner sep=2pt,fill=none,draw=none] at (2.5,1.005) {\tiny x};
	\node [inner sep=2pt,fill=none,draw=none] at (3,0.985) {\tiny x};
	\node [inner sep=2pt,fill=none,draw=none] at (3.5,0.98) {\tiny x};
	\node [inner sep=2pt,fill=none,draw=none] at (4,0.965) {\tiny x};
	\node [inner sep=2pt,fill=none,draw=none] at (4.37,1.) {{\tiny Estimation }$R_c^{\cn}$};
	\end{scope}
    \end{tikzpicture}
    \caption{Graph summarizing the result obtained in the \name model. Estimation of the critical radius for different values of $p$ are given in Section~\ref{sec:Estimates}.}
    \label{fig:ResultsCN}
\end{figure}

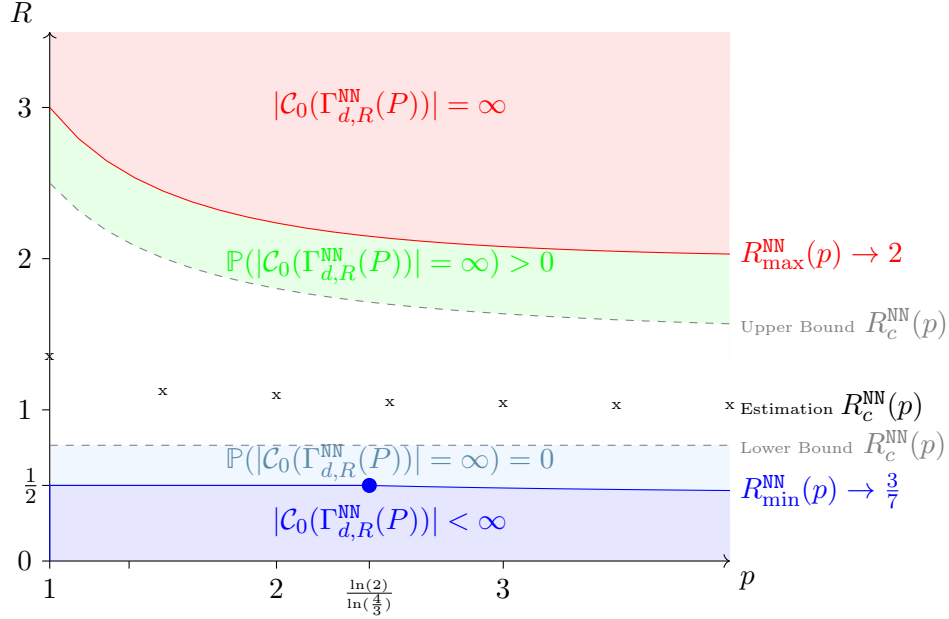
\begin{figure}
    \centering
    \begin{tikzpicture}[domain=0:10]
	\begin{scope}[x=3cm,y=2cm,xshift=-3cm]
	\fill[red!10] [smooth, samples = 200, domain={1}:{4}]
    	(1,3.5) -- plot (\x,{(2^(\x) + 1)^(1/(\x))}) -- (4,3.5) -|  cycle;
	\fill[green!10] [smooth, samples = 200, domain={1}:{4}]
    	(1,1) -- plot (\x,{(2^(\x) + 1)^(1/(\x))}) -- (4,1) -|  cycle;
	\fill[white] [smooth, samples = 200, domain={1}:{4}]
    	(1,1) -- plot (\x,{((3/2)^(\x) + 1)^(1/(\x))}) -- (4,1) -|  cycle;
	\fill[aliceblue] (1,0) -- (1,{(2-2^(1/2))^(1/2)}) -- (4,{(2-2^(1/2))^(1/2)}) -- (4,0) -|  cycle;
    \fill[blue!10] [smooth]
    	(1,0) -- (1,1/2) -- (2.4094,1/2) -- (2.5,0.4968) --	(2.6,0.4936) -- (2.7,0.4907) --	(2.8,0.488) -- (2.9,0.4854)	-- (3,0.4831) -- (3.1,0.4809) -- (3.2,0.4788) -- (3.3,0.4769) -- (3.4,0.4751) -- (3.5,0.4734) -- (3.6,0.4718) -- (3.7,0.4703) -- (3.8,0.4689) -- (3.9,0.4675) -- (4,0.4663) -- (4,0) -|  cycle;
	\node [draw=none,fill=none] at (2.5,3) {\color{red}$|\C_0(\Gamma_{d,R}^{\nn}(P))| = \infty$};
	\node [draw=none,fill=none] at (2.5,1.95) {\color{green}$\mathbb{P}(|\C_0(\Gamma_{d,R}^{\nn}(P))| = \infty)>0$};
	\node [draw=none,fill=none] at (2.5,0.65) {\color{airforceblue}$\mathbb{P}(|\C_0(\Gamma_{d,R}^{\nn}(P))| = \infty)=0$};
	\node [draw=none,fill=none] at (2.5,.25) {\color{blue}$|\C_0(\Gamma_{d,R}^{\nn}(P))| < \infty$};
	\end{scope}

	\begin{scope}[y=2cm]
	\draw [->] (0,0) -- (0,3.5);
	\node [anchor=south east,fill=white,draw=none] at (-.1,3.5) {$R$};
	\node [anchor= east,fill=none,draw=none] at (-.1,0) {0};
	\node [anchor=east,fill=none,draw=none] at (0,1/2) {$\frac{1}{2}$};
	\node [anchor=east,fill=white,draw=none] at (-.1,1) {$1$};
	\node [anchor=east,fill=white,draw=none] at (-.1,2) {$2$};
	\node [anchor=east,fill=white,draw=none] at (-.1,3) {$3$};
	\draw (-0.1,0) -- (0,0);
	\draw (-0.1,1/2) -- (0,1/2);
	\draw (-0.1,1) -- (0,1);
	\draw (-0.1,2) -- (0,2);
	\draw (-0.1,3) -- (0,3);
	\end{scope}
	
	\begin{scope}[x=3cm,xshift=-3cm]
		\draw [->] (1,0) -- (4,0);
		 \node [anchor=north west,fill=white,draw=none] at (4,0) {$p$};
		 \node [anchor=north,fill=white,draw=none] at (1,-.1) {$1$};
		 \node [anchor=north,fill=white,draw=none] at (2,-.1) {$2$};
		 \node [anchor=north,fill=white,draw=none] at (3,-.1) {$3$};
		 \node [anchor=north,fill=white,draw=none] at (2.4094,-.1) {\tiny$\frac{\ln(2)}{\ln(\frac{4}{3})}$};
		 \draw (1,-0.1) -- (1,0);
		 \draw (1.35,-0.1) -- (1.35,0);
		 \draw (2,-0.1) -- (2,0);
		 \draw (3,-0.1) -- (3,0);
		 \draw (2.4094,-0.1) -- (2.4094,0);
	\end{scope}
	
	\begin{scope}[x=3cm, y=2cm, xshift=-3cm]
	\draw [color=blue] (1,1/2) -- (2.4094,1/2);
	\node [fill=blue,inner sep=2pt,shape=circle] at (2.4094,1/2) {};
	\draw[color=blue] (1,0) -- (1,1/2) -- (2.4094,1/2) -- (2.5,0.4968) --	(2.6,0.4936) -- (2.7,0.4907) --	(2.8,0.488) -- (2.9,0.4854)	-- (3,0.4831) -- (3.1,0.4809) -- (3.2,0.4788) -- (3.3,0.4769) -- (3.4,0.4751) -- (3.5,0.4734) -- (3.6,0.4718) -- (3.7,0.4703) -- (3.8,0.4689) -- (3.9,0.4675) -- (4,0.4663)node[right,fill=white,draw=none] {$R_{\min}^{\nn}(p)\to\frac{3}{7}$};
	
	\draw[color=red,domain=1:4] plot (\x,{(2^(\x) + 1)^(1/(\x))}) node[right,fill=white,draw=none] {$R_{\max}^{\nn}(p)\to2$};
	\draw[color=gray,dashed,domain=1:4] plot (\x,{((3/2)^(\x) + 1)^(1/(\x))}) node[right,fill=none,draw=none] {{\tiny Upper Bound} $R_c^{\nn}(p)$};
	\draw[color=gray,dashed] (1,{(2-2^(1/2))^(1/2)}) -- (4,{(2-2^(1/2))^(1/2)}) node[right,fill=none,draw=none] {{\tiny Lower Bound} $R_c^{\nn}(p)$};
	
	\node [inner sep=2pt,fill=none,draw=none] at (1,1.35) {\tiny x};
	\node [inner sep=2pt,fill=none,draw=none] at (1.5,1.12) {\tiny x};
	\node [inner sep=2pt,fill=none,draw=none] at (2,1.095) {\tiny x};
	\node [inner sep=2pt,fill=none,draw=none] at (2.5,1.05) {\tiny x};
	\node [inner sep=2pt,fill=none,draw=none] at (3,1.04) {\tiny x};
	\node [inner sep=2pt,fill=none,draw=none] at (3.5,1.03) {\tiny x};
	\node [inner sep=2pt,fill=none,draw=none] at (4,1.025) {\tiny x};
	\node [inner sep=2pt,fill=none,draw=none] at (4.45,1.025) {{\tiny Estimation }$R_c^{\nn}(p)$};
	\end{scope}
    \end{tikzpicture}
    \caption{Graph summarizing the result obtained in the nearest neighbor model. Estimation of the critical radius for different values of $p$ are given in Section~\ref{sec:Estimates}.}
    \label{fig:ResultsNN}
\end{figure}

\paragraph{Outline of the article.}

Section~\ref{sec:Rmax} is devoted to the proof of Theorem~\ref{thm:Rmax}. Precisely, we show that all points are connected when $R\geq\|(2,1)\|_p$, while if $R<\|(2,1)\|_p$, we can construct an event of positive probability for which the origin belongs to a finite connected component. In Section~\ref{sec:Rmin}, we prove Theorem~\ref{thm:Rmin}, using self-avoiding paths. In order to bound $R_{\min}^{\cn}$ and $R_{\min}^{\nn}$ by above, it is sufficient to exhibit a single configuration for which there is an infinite path. For the lower bound, we provide firstly a proof quite simple for the case $p=\infty$, then a second more general proof for any $\mathcal{L}^p-$norms. The strategy is to assume that there exists a configuration with an infinite self-avoiding path, and to show that a certain quantity decreases strictly along this path,  which then leads to a contradiction. Section~\ref{sec:Rc} is devoted to the proof of the boundaries given in Theorem~\ref{thm:Rc}. We define a variant of our nearest-neighbor model, and compare it with the Bernoulli percolation, with the help of a duality argument.

\section{Proof of Theorem~\ref{thm:Rmax} ($R_{\max}$)}\label{sec:Rmax}
Let $p\in[1,\infty]$ be fixed. The following proof is treated for the case of the \name model. It can be adapted for the nearest neighbor model without major changes.

\paragraph{Upper bound.}
Let $R\geq\|(2,1)\|_p$. For any $P_{0,0} \in [0,1]^2$ and for any $P_{1,0} \in [1,2]\times[0,1]$, $\|P_{1,0}-P_{0,0}\|_p\leq \|(2,1)\|_p$. It follows that the points $(0,0)$ and $(1,0)$ are neighbors in $\Gamma_{d,R}^{\cn}(P)$. For the same reason, the points $(-1,0)$, $(0,1)$ and $(0,-1)$ are also neighbors of $(0,0)$. By induction, all the points are connected to $(0,0)$. 

\paragraph{Lower Bound.}
Let $\eps>0$, and consider a radius $R=\|(2,1)\|_p-\eps$.
We show that there exists a collection $P$ of points such that $|\C_0(\Gamma_{p,R}^{\cn}(P))| < \infty$.

For that we set for any $(i,j)\in \{-1,0\}^2$, $P_{i,j} =(0,0),$ and we define
$P_{0,1} = (1,2)$, $P_{0,2} = (1,3)$, $P_{1,0} = (2,1)$, $P_{1,1} = (2,2)$, $P_{2,0} = (3,1)$.

In the other three quadrants, we similarly define three sets of five points by rotation, see Figure~\ref{fig:Rmax}.

The set $\C_0(\Gamma_{d,R}^{\cn}(P))$ is then reduced to the four points  $(-1,-1)$, $(-1,0)$, $(0,-1)$ and $(0,0)$ which concludes the proof.

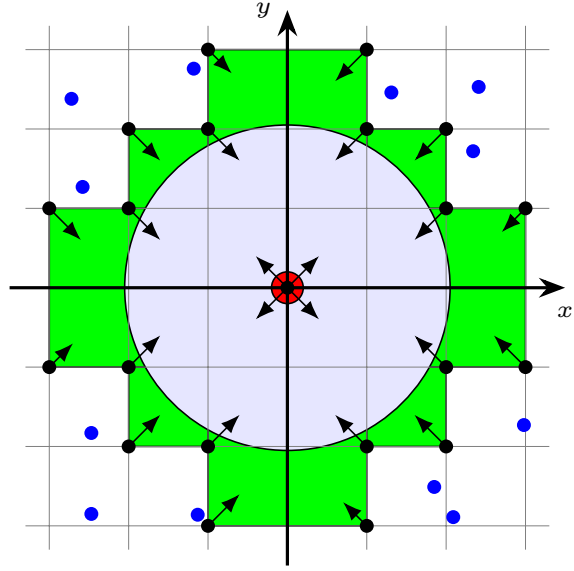
\begin{figure}[ht!]
    \centering
    \scalebox{1.5}{
    \begin{tikzpicture}[scale = .7,>=Latex]
        \filldraw[fill=green] (-2,-2) rectangle (2,2);
        \filldraw[fill=green] (-3,-1) rectangle (3,1);
        \filldraw[fill=green] (-1,-3) rectangle (1,3);
        \filldraw[fill=blue!10] (0,0) circle [radius=2.05];
        \filldraw[fill=red] (0,0)  circle [radius=0.2];
        \draw[step=1cm,color=gray,ultra thin] (-3.3,-3.3) grid (3.3,3.3);
        \node[inner sep=0pt,minimum size=3.5pt,fill=black,shape=circle] at (0,0) {};
        \foreach \i in{0,...,3}{
            \begin{scope}[rotate=\i*90]
                \node[inner sep=0pt,minimum size=3.5pt,fill=black,shape=circle] at (2,1) {};
                \draw[->] (2,1) -- (1.6,0.6);
                \node[inner sep=0pt,minimum size=3.5pt,fill=black,shape=circle] at (1,2) {};
                \draw[->] (1,2) -- (.6,1.6);
                \node[inner sep=0pt,minimum size=3.5pt,fill=black,shape=circle] at (2,2) {};
                \draw[->] (2,2) -- (1.6,1.6);
                \node[inner sep=0pt,minimum size=3.5pt,fill=black,shape=circle] at (1,3) {};
                \draw[->] (1,3) -- (.6,2.6);
                \node[inner sep=0pt,minimum size=3.5pt,fill=black,shape=circle] at (3,1) {};
                \draw[->] (3,1) -- (2.7,0.7);
                \draw[->] (0,0) -- (.4,.4);
            \end{scope}
        }
    
    \node [inner sep=0pt,minimum size=3.5pt,fill=blue,shape=circle] at  (-2.47, -2.85) {};
    \node [inner sep=0pt,minimum size=3.5pt,fill=blue,shape=circle] at  (-2.47, -1.83) {};
    \node [inner sep=0pt,minimum size=3.5pt,fill=blue,shape=circle] at  (-1.13, -2.86) {};
    \node [inner sep=0pt,minimum size=3.5pt,fill=blue,shape=circle] at  (-2.58, 1.27) {};
    \node [inner sep=0pt,minimum size=3.5pt,fill=blue,shape=circle] at  (-2.72, 2.38) {};
    \node [inner sep=0pt,minimum size=3.5pt,fill=blue,shape=circle] at  (-1.18, 2.76) {};
    \node [inner sep=0pt,minimum size=3.5pt,fill=blue,shape=circle] at  (1.31, 2.46) {};
    \node [inner sep=0pt,minimum size=3.5pt,fill=blue,shape=circle] at  (2.34, 1.72) {};
    \node [inner sep=0pt,minimum size=3.5pt,fill=blue,shape=circle] at  (2.41, 2.53) {};
    \node [inner sep=0pt,minimum size=3.5pt,fill=blue,shape=circle] at  (2.98, -1.73) {};
    \node [inner sep=0pt,minimum size=3.5pt,fill=blue,shape=circle] at  (2.09, -2.89) {};
    \node [inner sep=0pt,minimum size=3.5pt,fill=blue,shape=circle] at  (1.85, -2.51) {};
    
        \draw [->,>=Stealth,thick] (-3.5,0) -- (3.5,0);
        \draw [->,>=Stealth,thick] (0,-3.5) -- (0,3.5);
        \node[draw=none,fill=none] at (-.3,3.5) {\tiny $y$};
        \node[draw=none,fill=none] at (3.5,-.3) {\tiny $x$};
    \end{tikzpicture}
    }
    \caption{Configuration where $|\C_0(\Gamma_{p,R}^{\cn}(P))|<\infty$ for $p=2$. The arrows indicate to which cell the points belong. The red area contains the finite connected component $\C_0(\Gamma_{p,R}^{\cn}(P))$. The blue area combined with the red one correspond to the disk  of radius $\|(2,1)\|_p-\eps$. The points of the cells that intersect this disk are placed in the green area to avoid any connection with the points inside the red area.}
    \label{fig:Rmax}
\end{figure}

\section{Proof of Theorem~\ref{thm:Rmin} ($R_{\min}$)}\label{sec:Rmin}

In this section, we determine the values of $R_{\min}^{\cn}(p)$ and $R_{\min}^{\nn}(p)$ for all $p\in[1,\infty]$. To do this, we use the fact that the connected component of the origin is infinite if and only if it contains an infinite self-avoiding path starting from the origin, where a \emph{self-avoiding path} $(u_k)_{0\leq k < n}$, with $n\in\NN\cup\{\infty\}$ (where, to be precise, $\NN = \{0,1,\dots\}$ denotes the set of non-negative integers), is a path that does not visit a vertex more than once, i.e.\ for any $k\neq l$, $u_k \neq u_l$. For $k\in\NN$, we denote by $x_k$ (resp. $y_k$) the horizontal (resp.\ vertical) coordinate of $u_k$.

\subsection{Upper bounds}\label{sec:UBR_min}
Let $p\in[1,\infty]$.
\paragraph{Proof of \boldmath $R_{\min}^{\cn}(p)$ and $R^{\nn}_{\min}(p) \leq \displaystyle\frac{1}{2}$.}

We set, for $k\in\NN$,
$$\begin{array}{ll}
P_{2k,0}=(2k+\frac{1}{2},0),  & P_{2k+1,0}=(2k+1,0),\\
P_{2k+1,-1}=(2k+\frac{3}{2},0), & P_{2k+2,-1}=(2k+2,0),
\end{array}$$
see Figure~\ref{fig:SAP1/2}. The sequence $u = (u_k)_{k \geq 0}$ defined, for any $k\in\NN$, by
$$\begin{array}{ll}
u_{4k}=(2k,0),  & u_{4k+1}=(2k+1,0),\\
u_{4k+2}=(2k+1,-1), & u_{4k+3}=(2k+2,-1)
\end{array}$$
is then an infinite self-avoiding path of $\Gamma_{p,R}^{\cn}(P)$ and of $\Gamma_{p,R}^{\nn}(P)$.

\begin{figure}[ht!]
	\centering
	\begin{tikzpicture}[>=Latex]
		\draw[step=1cm,color=gray,ultra thin] (.2,-.4) grid (4.4,.4);
		\foreach \i in {0}{
			\begin{scope}[xshift=\i*2cm]
				\node [fill=red, inner sep=2pt,shape=circle] at (0.5,0) {};
				\node [fill=red, inner sep=2pt,shape=circle] at (1,0) {};
				\node [fill=red, inner sep=2pt,shape=circle] at (1.5,0) {};
				\node [fill=red, inner sep=2pt,shape=circle] at (2,0) {};
				\draw [red] (0.5,0) -- (2.5,0);
				\draw [thick, ->, color=red] (0.5,0) -- (0.5,0.4);
				\draw [thick, ->, color=red] (1,0) -- (1.4,0.4);
				\draw [thick, ->, color=red] (1.5,0) -- (1.5,-0.4);
				\draw [thick, ->, color=red] (2,0) -- (2.4,-0.4);
			\end{scope}
		}
		\foreach \i in {1}{
			\begin{scope}[xshift=\i*2cm]
				\node [fill=black, inner sep=2pt,shape=circle] at (0.5,0) {};
				\node [fill=black, inner sep=2pt,shape=circle] at (1,0) {};
				\node [fill=black, inner sep=2pt,shape=circle] at (1.5,0) {};
				\node [fill=black, inner sep=2pt,shape=circle] at (2,0) {};
				\draw (0.5,0) -- (2,0);
				\draw [thick, ->, color=black] (0.5,0) -- (0.5,0.4);
				\draw [thick, ->, color=black] (1,0) -- (1.4,0.4);
				\draw [thick, ->, color=black] (1.5,0) -- (1.5,-0.4);
				\draw [thick, ->, color=black] (2,0) -- (2.4,-0.4);
			\end{scope}
		}
	\end{tikzpicture}
	\caption{Construction of an infinite self-avoiding path of $\Gamma_{p,R}^{\cn}(P)$ and of $\Gamma_{p,R}^{\nn}(P)$ for $R\geq1/2$ and for any $p\in[1,\infty]$. The pattern that is repeated periodically is in red and the arrows indicate to which cell the points belong.}
	\label{fig:SAP1/2}
\end{figure}
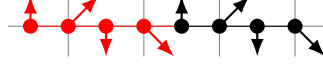

\paragraph{Proof of \boldmath $R_{\min}^{\cn}(p) \leq \displaystyle \frac{2}{1+2^{2-\frac{1}{p}}}$.} 

We denote 
$$f(p)=\left\{
\begin{array}{ll}
     \displaystyle \frac{2}{1+2^{2-\frac{1}{p}}} & \text{if } p\in[1,\infty), \\[0.5cm]
     \displaystyle \frac{2}{5} & \text{if } p=\infty.
\end{array}
\right.
$$

Consider the sequence $u = (u_k)_{k \geq 0}$ defined, for any $k\geq 0$, by
$$\begin{array}{lll}
u_{6k}=(2k,2k),  & u_{6k+1}=(2k+1,2k-1), & u_{6k+2}=(2k+1,2k),\\ u_{6k+3}=(2k+2,2k), & u_{6k+4}=(2k+1,2k+1), & u_{6k+5}=(2k+2,2k+1).
\end{array}$$

Let set the points of the concerned cells in such a way, for any $k\geq 0$,
$$\begin{array}{ll}
P_{2k,2k}=\left(2k+1-\frac{f(p)}{2},2k\right),  & P_{2k+1,2k-1}=\left(2k+1+\frac{f(p)}{2},2k\right), \\
P_{2k+1,2k}=\left(2k+1+\frac{2+f(p)}{4}, 2k + \frac{2-f(p)}{4}\right), &P_{2k+2,2k}=\left(2k+2,2k+1-\frac{f(p)}{2}\right), \\
P_{2k+1,2k+1}=\left(2k+2,2k+1+\frac{f(p)}{2}\right), & P_{2k+2,2k+1}=\left(2k+2 +\frac{2-f(p)}{4},2k+1+\frac{2+f(p)}{4}\right),
\end{array}$$
see Figure~\ref{fig:SAP<1/2}.
Now, let us compute the distances between two consecutive points along the sequence~$u$. Using the symmetries, it is sufficient to compute only the two following distances:
\begin{align*}
    \|P_{0,0} - P_{1,-1}\|_p & = \|(f(p),0)\|_p \text{ and } \\
    \|P_{1,0} - P_{1,-1}\|_p & = \bigg\|\bigg(\frac{2+f(p)}{4}-\frac{f(p)}{2},\frac{2-f(p)}{4}\bigg)\bigg\|_p\\
    & = \bigg\|\bigg(\frac{2-f(p)}{4},\frac{2-f(p)}{4}\bigg)\bigg\|_p.
\end{align*}    

For any $p\in[1,\infty]$, those two distances are equal to $f(p)$.

So $u$ is an infinite self-avoiding path of $\Gamma_{p,R}^{\cn}(P)$ for $R\geq f(p)$. On the other hand, it is not a path of $\Gamma_{p,R}^{\nn}(P)$ since the edge $(u_0,u_1) = \left((0,0),(1,-1)\right)$ is diagonal which is not an allowed edge.

\begin{figure}[ht!]
	\centering
	\begin{tikzpicture}[scale=1.2,>=Latex]
		\draw[step=1cm,color=gray,ultra thin] (-.3,-.3) grid (5.3,4.3);
		\node (A) [fill=red,inner sep=2pt,shape=circle] at (1-2/10,0) {};
		\draw [thick,->, color=red] (1-2/10,0) -- (1-2/10,0.4);
		\node (B) [fill=red,inner sep=2pt,shape=circle] at (1+2/10,0) {};
		\draw [thick,->, color=red] (1+2/10,0) -- (1+2/10,-0.4);
	 	\node (C) [fill=red,inner sep=2pt,shape=circle] at (1 + 1/2 + 2/20,1/2 - 2/20) {};
	 	\node (D) [fill=red,inner sep=2pt,shape=circle] at (2,1-2/10) {};
	  	\draw [thick, ->, color=red] (2,1-2/10) -- (2.4,1-2/10);
	  	\node (E) [fill=red,inner sep=2pt,shape=circle] at (2,1+2/10) {};
	  	\draw [thick, ->, color=red] (2,1+2/10) -- (1.6,1+2/10);
	  	\node (F) [fill=red,inner sep=2pt,shape=circle] at (2+1/2-2/20,1+1/2+2/20) {};
		\node (G) [fill=none,inner sep=2pt,shape=circle] at (3-2/10,2) {}; 
		\draw [color=red] (A) -- (B);
	  	\draw [color=red] (C) -- (B);
	  	\draw [color=red] (C) -- (D);
	  	\draw [color=red] (E) -- (D);
	  	\draw [color=red] (E) -- (F);
	  	\draw [color=red] (G) -- (F);
		
		\begin{scope}[xshift=2cm,yshift=2cm]
		\node (A) [fill=black,inner sep=2pt,shape=circle] at (1-2/10,0) {};
		\draw [thick,->, color=black] (1-2/10,0) -- (1-2/10,0.4);
		\node (B) [fill=black,inner sep=2pt,shape=circle] at (1+2/10,0) {};
		\draw [thick,->, color=black] (1+2/10,0) -- (1+2/10,-0.4);
	 	\node (C) [fill=black,inner sep=2pt,shape=circle] at (1 + 1/2 + 2/20,1/2 - 2/20) {};
	 	\node (D) [fill=black,inner sep=2pt,shape=circle] at (2,1-2/10) {};
	  	\draw [thick, ->, color=black] (2,1-2/10) -- (2.4,1-2/10);
	  	\node (E) [fill=black,inner sep=2pt,shape=circle] at (2,1+2/10) {};
	  	\draw [thick, ->, color=black] (2,1+2/10) -- (1.6,1+2/10);
	  	\node (F) [fill=black,inner sep=2pt,shape=circle] at (2+1/2-2/20,1+1/2+2/20) {};
		\node (G) [fill=black,inner sep=2pt,shape=circle] at (3-2/10,2) {}; 
	  	\draw [thick, ->, color=black] (3-2/10,2) -- (3-2/10,2.4);
		\draw  (A) -- (B);
	  	\draw  (C) -- (B);
	  	\draw  (C) -- (D);
	  	\draw (E) -- (D);
	  	\draw (E) -- (F);
	  	\draw (G) -- (F);
		\end{scope}
	\end{tikzpicture}
	\caption{Construction of an infinite self-avoiding path of $\Gamma_{p,R}^{\cn}(p)$ for $p=\infty$ and $R\geq 2/5$. The pattern that is repeated periodically is in red and the arrows indicate to which cell the points belong.}
	\label{fig:SAP<1/2}
\end{figure}
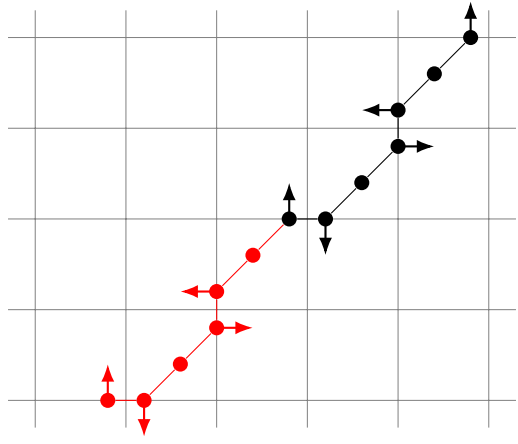

\paragraph{Proof of the upper bound for \boldmath$R^{\nn}_{\min}(p)$.}

Unlike the two previous cases, the upper bound is a non explicit value defined as the unique solution of the equation
\begin{equation}\label{eq:RMINLocal}
    (2-R)^p + (3-3R)^p = (4R)^p
\end{equation}
for $R\in[0,1]$.

\begin{lemma}\label{lem:R(p)}

Now, we present a self-avoiding path which gives us the lower bound for the nearest neighbor model when $p\in \left[\ln(2)/\ln(4/3),\infty\right]$. 
For any $p \in [1,\infty)$, Equation~\eqref{eq:RMINLocal} has a unique solution in $[0,1]$, denoted by $R(p)$. Furthermore, the function $p\mapsto R(p)$ is decreasing on $[1,\infty)$; $R(p)< 1/2$ for $p> \ln(2)/\ln(4/3)$; and $\lim_{p\to\infty} R(p)= 3/7$.
\end{lemma}

\begin{proof}
    First, we prove that  Equation~\eqref{eq:RMINLocal} has a unique solution on the interval $[0,1]$. For any $p \in [1,\infty)$, consider the function on $[0,1]$,
    \begin{displaymath}
    g_p(R) = (4R)^p - (2-R)^p - (3-3R)^p = (4R)^p - \|2-R,3-3R\|_p^p.
    \end{displaymath}
    
	Its derivative is
	$$g_p'(R) = 4^p p R^{p-1} + p (2-R)^{p-1} +3^p p(1-R)^{p-1}.$$
    It is positive on $[0,1]$. Hence, $g_p$ is continuous and increasing, with $g_p(0)= -2^p -3^p<0$ and $g_p(1) = 4^p - 1 >0$. Hence, $g_p(R)=0$ has a unique solution on $[0,1]$.
    
	Now, let us prove that $p\mapsto R(p)$ is decreasing on $[1,\infty)$.
	For any $R\in[0,1]$ and $1\leq p<q$, $\|2-R,3-3R\|_p > \|2-R,3-3R\|_q$, by the inclusion of the ball for the $\mathcal{L}^p-$norm in the ball for the $\mathcal{L}^q-$norm.
	Hence, $4R(p) = \|2-R(p),3-3R(p)\|_p > \|2-R(p),3-3R(p)\|_q$ which implies that $g_q(R(p))>0$, and so $R(q) < R(p)$.
	
	Let us now prove that $\displaystyle R \left(\frac{\ln(2)}{\ln\left(\frac{4}{3}\right)}\right)=\frac{1}{2}$. Let $p\in[1,\infty)$ be such that $R(p)=\frac{1}{2}$. Then,  $\left(\frac{3}{2}\right)^p + \left(\frac{3}{2}\right)^p = 2^p$, so that $2 = \left(\frac{4}{3}\right)^p$.
	
	Since $p\mapsto R(p)$ is decreasing on $[1,\infty)$ with $R(p) > 0$, the limit $L=\lim_{p\to\infty}R(p)$ exists. It satisfies 
	$4L = \|2-L,3-3L\|_{\infty}$. By the above, $L<1/2$, and if $R<1/2$, $2-R > 3-3R$, then $4L = 3-3L$, which implies that $L = 3/7$.
\end{proof}

We denote $R(\infty)=3/7$. For any $p \in [1,\infty]$, we set the points:
$$\begin{array}{ll}
P_{0,0}=\left(\frac{R(p)}{2},0\right),  & P_{0,-1}=\left(\frac{1}{2}+\frac{R(p)}{4},-\frac{3}{4} + \frac{3R(p)}{4}\right),\\
P_{1,-1}=\left(1,-\frac{3}{2} + \frac{3R(p)}{2}\right), & P_{1,-2}=\left(1,-\frac{3}{2}+ \frac{R(p)}{2}\right).
\end{array}$$
which are represented by the red pattern in Figure~\ref{fig:SAP4V}. From these 4 points, we apply 3 transformations that are combinations of symmetries and rotations to obtain the 3 other colored patterns in Figure~\ref{fig:SAP4V}. Finally, the remaining points are obtained by applying translations of vector $(0,4k)$ for $k\in\NN$ to those 16 points.

Now, let us check that the distances between two consecutive points along the sequence thus obtained are less than $R(p)$. For that, we only need to compute the two following distances
\begin{align*}
	\|P_{1,-1} - P_{1,-2}\|_p & = \|(0,R(p))\|_p = R(p) \text{ and } \\
	\|P_{0,-1} - P_{0,0}\|_p & = \left\|\left(-\frac{1}{2}+\frac{R(p)}{4},-\frac{3}{4}+\frac{3R(p)}{4}\right)\right\|_p \\
	& = \left\{
	\begin{array}{ll}
		\displaystyle \left( \left(\frac{2-R(p)}{4}\right)^p + \left(\frac{3-3R(p)}{4}\right)^p\right)^{\frac{1}{p}} &\mbox{ if } p\in[1,\infty),\\[0.5cm]
		\displaystyle \max\left\{\frac{11}{28},\frac{12}{28}\right\} &\mbox{ if } p=\infty.\\
	\end{array}
	\right.
	\\
	& = R(p).
\end{align*}


The path $u$ obtained from those points is thus an infinite self-avoiding path of $\Gamma_{p,R}^{\nn}(P)$ for $R\geq R(p)$, see Figure~\ref{fig:SAP4V}.

\begin{figure}[ht!]
    \centering
    \begin{tikzpicture}[scale=1,>=Latex]
    \begin{scope}[rotate=-90]
    \draw[step=1cm,color=gray,ultra thin] (-.4,-.4) grid (5.4,8.4);
        \foreach \i in {0}{
        \node (A) [fill=red,inner sep=2pt,shape=circle] at (1,\i*4+3/14) {};
	    \draw [thick, ->, color=red] (1,\i*4+3/14) -- (.6,\i*4+3/14);
        \node (B) [fill=red,inner sep=2pt,shape=circle] at (10/7,\i*4+17/28) {};
        \node (C) [fill=red,inner sep=2pt,shape=circle] at (13/7,\i*4+1) {};
        \draw [thick, ->, color=red] (13/7,\i*4+1) -- (13/7,\i*4+1.4);
        \node (D) [fill=red,inner sep=2pt,shape=circle] at (16/7,\i*4+1) {};
        \draw [thick, ->, color=red] (16/7,\i*4+1) -- (16/7,\i*4+1.4);
        \node (E) [fill=darkpastelgreen,inner sep=2pt,shape=circle] at (19/7,\i*4+1) {};
        \draw [thick, ->, color=darkpastelgreen] (19/7,\i*4+1) -- (19/7,\i*4+.6);
        \node (F) [fill=darkpastelgreen,inner sep=2pt,shape=circle] at (22/7,\i*4+1) {};
        \draw [thick, ->, color=darkpastelgreen] (22/7,\i*4+1) -- (22/7,\i*4+.6);
        \node (G) [fill=darkpastelgreen,inner sep=2pt,shape=circle] at (25/7,\i*4+ 39/28) {};
        \node (H) [fill=darkpastelgreen,inner sep=2pt,shape=circle] at (4,\i*4+25/14) {};
        \draw [thick, ->, color=darkpastelgreen] (4,\i*4+25/14) -- (4.4,\i*4+25/14);
        \node (I) [fill=black,inner sep=2pt,shape=circle] at (4,\i*4+31/14) {};
        \draw [color=red] (A) -- (B);
        \draw [color=red] (C) -- (B);
        \draw [color=red] (C) -- (D);
        \draw [color=red] (E) -- (D);
        \draw [color=darkpastelgreen] (E) -- (F);
        \draw [color=darkpastelgreen] (G) -- (F);
        \draw [color=darkpastelgreen] (G) -- (H);
        \draw [color=darkpastelgreen] (I) -- (H);
        \begin{scope}[yscale=1,xscale=-1,xshift=-5cm,yshift=2cm]
        \node (A) [fill=blue,inner sep=2pt,shape=circle] at (1,\i*4+3/14) {};
	    \draw [thick, ->, color=blue] (1,\i*4+3/14) -- (.6,\i*4+3/14);
        \node (B) [fill=blue,inner sep=2pt,shape=circle] at (10/7,\i*4+17/28) {};
        \node (C) [fill=blue,inner sep=2pt,shape=circle] at (13/7,\i*4+1) {};
        \draw [thick, ->, color=blue] (13/7,\i*4+1) -- (13/7,\i*4+1.4);
        \node (D) [fill=blue,inner sep=2pt,shape=circle] at (16/7,\i*4+1) {};
        \draw [thick, ->, color=blue] (16/7,\i*4+1) -- (16/7,\i*4+1.4);
        \node (E) [fill=purple,inner sep=2pt,shape=circle] at (19/7,\i*4+1) {};
        \draw [thick, ->, color=purple] (19/7,\i*4+1) -- (19/7,\i*4+.6);
        \node (F) [fill=purple,inner sep=2pt,shape=circle] at (22/7,\i*4+1) {};
        \draw [thick, ->, color=purple] (22/7,\i*4+1) -- (22/7,\i*4+.6);
        \node (G) [fill=purple,inner sep=2pt,shape=circle] at (25/7,\i*4+ 39/28) {};
        \node (H) [fill=purple,inner sep=2pt,shape=circle] at (4,\i*4+25/14) {};
        \draw [thick, ->, color=purple] (4,\i*4+25/14) -- (4.4,\i*4+25/14);
        \node (I) [fill=black,inner sep=2pt,shape=circle] at (4,\i*4+31/14) {};
        \draw [color=blue] (A) -- (B);
        \draw [color=blue] (C) -- (B);
        \draw [color=blue] (C) -- (D);
        \draw [color=blue] (E) -- (D);
        \draw [color=purple] (E) -- (F);
        \draw [color=purple] (G) -- (F);
        \draw [color=purple] (G) -- (H);
        \draw [color=purple] (I) -- (H);
        \end{scope}
    }
        \foreach \i in {1}{
        \node (A) [fill=black,inner sep=2pt,shape=circle] at (1,\i*4+3/14) {};
	    \draw [thick, ->, color=black] (1,\i*4+3/14) -- (.6,\i*4+3/14);
        \node (B) [fill=black,inner sep=2pt,shape=circle] at (10/7,\i*4+17/28) {};
        \node (C) [fill=black,inner sep=2pt,shape=circle] at (13/7,\i*4+1) {};
        \draw [thick, ->, color=black] (13/7,\i*4+1) -- (13/7,\i*4+1.4);
        \node (D) [fill=black,inner sep=2pt,shape=circle] at (16/7,\i*4+1) {};
        \draw [thick, ->, color=black] (16/7,\i*4+1) -- (16/7,\i*4+1.4);
        \node (E) [fill=black,inner sep=2pt,shape=circle] at (19/7,\i*4+1) {};
        \draw [thick, ->, color=black] (19/7,\i*4+1) -- (19/7,\i*4+.6);
        \node (F) [fill=black,inner sep=2pt,shape=circle] at (22/7,\i*4+1) {};
        \draw [thick, ->, color=black] (22/7,\i*4+1) -- (22/7,\i*4+.6);
        \node (G) [fill=black,inner sep=2pt,shape=circle] at (25/7,\i*4+ 39/28) {};
        \node (H) [fill=black,inner sep=2pt,shape=circle] at (4,\i*4+25/14) {};
        \draw [thick, ->, color=black] (4,\i*4+25/14) -- (4.4,\i*4+25/14);
        \node (I) [fill=black,inner sep=2pt,shape=circle] at (4,\i*4+31/14) {};
        \draw [color=black] (A) -- (B);
        \draw [color=black] (C) -- (B);
        \draw [color=black] (C) -- (D);
        \draw [color=black] (E) -- (D);
        \draw [color=black] (E) -- (F);
        \draw [color=black] (G) -- (F);
        \draw [color=black] (G) -- (H);
        \draw [color=black] (I) -- (H);
        \begin{scope}[yscale=1,xscale=-1,xshift=-5cm,yshift=2cm]
        \node (A) [fill=black,inner sep=2pt,shape=circle] at (1,\i*4+3/14) {};
	    \draw [thick, ->, color=black] (1,\i*4+3/14) -- (.6,\i*4+3/14);
        \node (B) [fill=black,inner sep=2pt,shape=circle] at (10/7,\i*4+17/28) {};
        \node (C) [fill=black,inner sep=2pt,shape=circle] at (13/7,\i*4+1) {};
        \draw [thick, ->, color=black] (13/7,\i*4+1) -- (13/7,\i*4+1.4);
        \node (D) [fill=black,inner sep=2pt,shape=circle] at (16/7,\i*4+1) {};
        \draw [thick, ->, color=black] (16/7,\i*4+1) -- (16/7,\i*4+1.4);
        \node (E) [fill=black,inner sep=2pt,shape=circle] at (19/7,\i*4+1) {};
        \draw [thick, ->, color=black] (19/7,\i*4+1) -- (19/7,\i*4+.6);
        \node (F) [fill=black,inner sep=2pt,shape=circle] at (22/7,\i*4+1) {};
        \draw [thick, ->, color=black] (22/7,\i*4+1) -- (22/7,\i*4+.6);
        \node (G) [fill=black,inner sep=2pt,shape=circle] at (25/7,\i*4+ 39/28) {};
        \node (H) [fill=black,inner sep=2pt,shape=circle] at (4,\i*4+25/14) {};
        \draw [thick, ->, color=black] (4,\i*4+25/14) -- (4.4,\i*4+25/14);
        \node (I) [fill=black,inner sep=2pt,shape=circle] at (4,\i*4+31/14) {};
        \draw [thick, ->, color=black] (4,\i*4+31/14) -- (4.4,\i*4+31/14);
        \draw (A) -- (B);
        \draw (C) -- (B);
        \draw (C) -- (D);
        \draw (E) -- (D);
        \draw (E) -- (F);
        \draw (G) -- (F);
        \draw (G) -- (H);
        \draw (I) -- (H);
        \end{scope}
    }
    \end{scope}
    \end{tikzpicture}        
    \caption{Construction of an infinite self-avoiding path of $\Gamma_{p,R}^{\nn}(p)$ for the case $p=\infty$ and $R\geq3/7$. The pattern that is repeated periodically is represented in colors. The arrows indicate to which cell the points belong.}
    \label{fig:SAP4V}
\end{figure}

\subsection{Forbidden paths}\label{sec:FPaths}
We now address the lower bound. Since  $R_{\min}^{\cn}(p)\leq 1/2$ and $R_{\min}^{\nn}(p)\leq 1/2$ for any $p\geq1$, we assume that $R<1/2$. 
For a path $(u_k)_{0\leq k\leq n}$, we denote by $s_k = u_k - u_{k-1}$ the step from $u_{k-1}$ to $u_{k}$, for $1\leq k\leq n$. Given that $R<1/2$, it follows that $$\forall k\in\{1,\ldots,n\}, \quad s_{k}\in\step
=\{\swa,\la,\nwa,\da,\ua,\sea,\ra,\nea\},$$ where each arrow represents the corresponding vector in $\{-1,0,1\}^2\setminus\{(0,0)\}$ (for example $\ra = (1,0)$ and $\swa =(-1,-1)$).\par

The following lemma gives a necessary condition for a path to belong to $\Gamma_{p,R}^{\cn}(P)$ or to $\Gamma^{\nn}_{p,R}(P)$.

\begin{lemma}\label{lem:R<p/q} 
    Let $a,b\in \ZZ_{>0} = \{1,2,\dots\}$, let $R \in (0, a/b)$, and let $u_0,\ldots,u_n\in\ZZ^2$. If there exists $P\in\mathcal P$ such that $u=(u_0,\ldots,u_n)$ is a path of $\Gamma_{p,R}^{\cn}(P)$ (resp.\ of $\Gamma^{\nn}_{p,R}(P)$), then 
    \begin{equation}
    \forall k\in \{0,\ldots, n-b\}, \quad \|u_{k+b}-u_k\|_\infty \leq a.
    \end{equation}
\end{lemma}

\begin{proof}
    It is enough to prove the result for $k=0$. Assume that for all $l\in\{0,\ldots,b-1\}$, $\|P_{u_{l+1}}-P_{u_{l}}\|_p\leq R$. 
    Then $|X_{u_{l+1}}-X_{u_{l}}|\leq R$, and it follows that 
    $|X_{u_b} - X_{u_0}|\leq  bR< a.$
    Consequently, $|x_b-x_0|<a+1$, so that $|x_b-x_0|\leq a$. Similarly, $|y_b - y_0|\leq a$.
\end{proof}

Lemma~\ref{lem:R<p/q} allows us to display patterns that cannot appear in a path. In the rest of the article, we apply it in three cases.
\begin{enumerate}[label=(\roman*)]
    \item $(a,b)=(1,2)$, see Figure~\ref{fig:R<1/2} for a list of forbidden patterns when $R<1/2$. \\
    This case is used for all the lower bounds of Section~\ref{sec:LBR_min}.
    \item $(a,b)=(2,5)$, see Figure~\ref{fig:R<2/5} for some examples of forbidden patterns when $R<2/5$.\\
    This case is used for the lower bound of $R_{\min}^{\cn}(\infty)$ in  Section~\ref{sec:LBR_minCNInfinity}.
    \item $(a,b)=(3,7)$, see Figure~\ref{fig:R<3/7} for some examples of forbidden patterns when $R<3/7$. \\
    This case is used for the lower bound of $R_{\min}^{\nn}(\infty)$ in Section~\ref{sec:LBR_minNNInfinity}.
\end{enumerate}

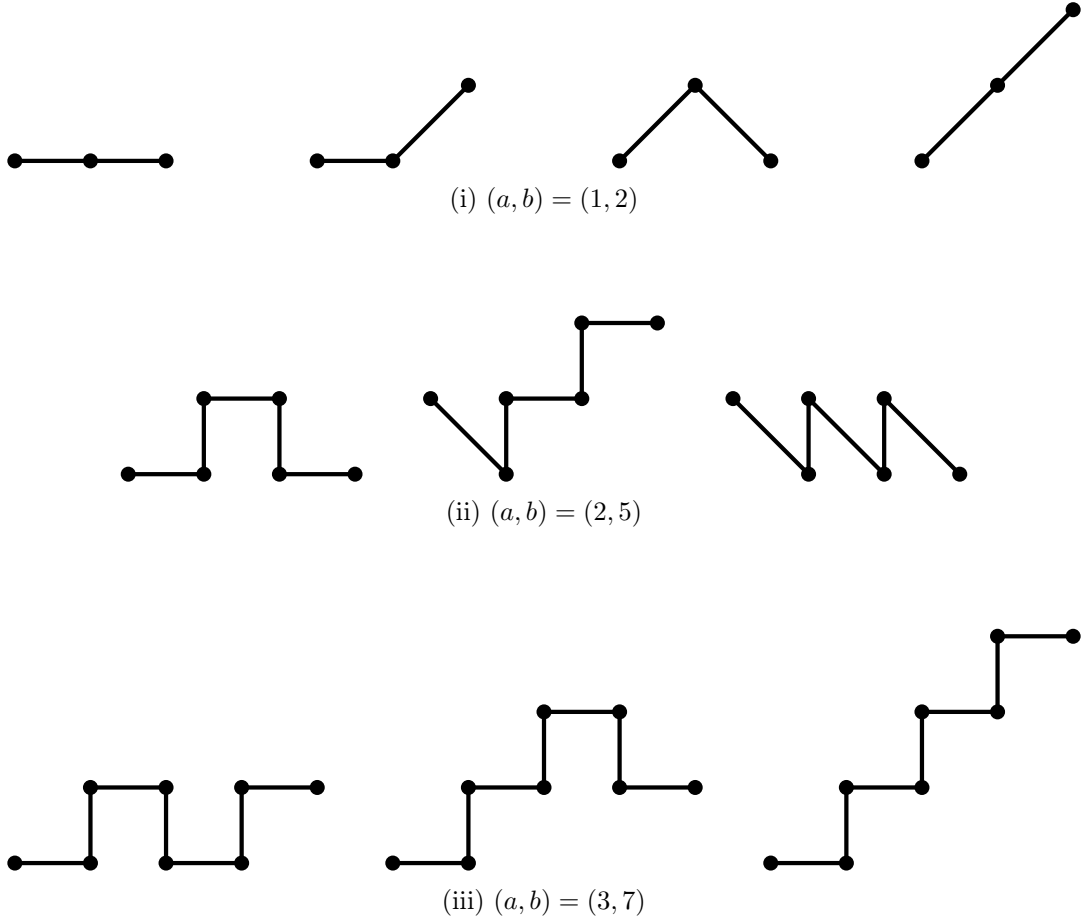
\begin{figure}[ht!]
    \centering
    \begin{subfigure}{\textwidth}
    \centering
	\begin{tikzpicture}
		\begin{scope}[xshift = 0cm]
			\node[fill=black,inner sep=2pt,shape=circle] at (0,0) {};
			\node[fill=black,inner sep=2pt,shape=circle] at (1,0) {};
			\node[fill=black,inner sep=2pt,shape=circle] at (2,0) {};
			\draw[ultra thick] (0,0) -- (1,0) -- (2,0);
		\end{scope}
		\begin{scope}[xshift = 4cm]
			\node[fill=black,inner sep=2pt,shape=circle] at (0,0) {};
			\node[fill=black,inner sep=2pt,shape=circle] at (1,0) {};
			\node[fill=black,inner sep=2pt,shape=circle] at (2,1) {};
			\draw[ultra thick] (0,0) -- (1,0) -- (2,1);
		\end{scope}
		\begin{scope}[xshift = 8cm]
			\node[fill=black,inner sep=2pt,shape=circle] at (0,0) {};
			\node[fill=black,inner sep=2pt,shape=circle] at (1,1) {};
			\node[fill=black,inner sep=2pt,shape=circle] at (2,0) {};
			\draw[ultra thick] (0,0) -- (1,1) -- (2,0);
		\end{scope}
		\begin{scope}[xshift = 12cm]
			\node[fill=black,inner sep=2pt,shape=circle] at (0,0) {};
			\node[fill=black,inner sep=2pt,shape=circle] at (1,1) {};
			\node[fill=black,inner sep=2pt,shape=circle] at (2,2) {};
			\draw[ultra thick] (0,0) -- (1,1) -- (2,2);
		\end{scope}
	\end{tikzpicture}
	\caption{$(a,b)=(1,2)$}
    \label{fig:R<1/2}
	\end{subfigure}
	\hfill
    \vspace{.8cm}
    \begin{subfigure}{\textwidth}
    \centering
    \begin{tikzpicture}
		\begin{scope}[ultra thick,inner sep =2pt,shape=circle,xshift=0cm]
			\node[fill=black] at (0,0) {};
			\node[fill=black] at (1,0) {};
			\node[fill=black] at (1,1) {};
			\node[fill=black] at (2,1) {};
			\node[fill=black] at (2,0) {};
			\node[fill=black] at (3,0) {};
			\draw (0,0) -- (1,0) -- (1,1) -- (2,1) -- (2,0) -- (3,0);
		\end{scope}
		\begin{scope}[ultra thick,inner sep =2pt,shape=circle,xshift=4cm]
			\node[fill=black] at (0,1) {};
			\node[fill=black] at (1,0) {};
			\node[fill=black] at (1,1) {};
			\node[fill=black] at (2,1) {};
			\node[fill=black] at (2,2) {};
			\node[fill=black] at (3,2) {};
			\draw (0,1) -- (1,0) -- (1,1) -- (2,1) -- (2,2) -- (3,2);
		\end{scope}
		\begin{scope}[ultra thick,inner sep =2pt,shape=circle,xshift=8cm]
			\node[fill=black] at (0,1) {};
			\node[fill=black] at (1,0) {};
			\node[fill=black] at (1,1) {};
			\node[fill=black] at (2,0) {};
			\node[fill=black] at (2,1) {};
			\node[fill=black] at (3,0) {};
			\draw (0,1) -- (1,0) -- (1,1) -- (2,0) -- (2,1) -- (3,0);
		\end{scope}
	\end{tikzpicture}
    \caption{$(a,b)=(2,5)$}
    \label{fig:R<2/5}
    \end{subfigure}
    \hfill
    \vspace{.8cm}
    \begin{subfigure}{\textwidth}
    \centering
    \begin{tikzpicture}
		\begin{scope}[ultra thick,inner sep =2pt,shape=circle,xshift=0cm]
			\node[fill=black] at (0,0) {};
			\node[fill=black] at (1,0) {};
			\node[fill=black] at (1,1) {};
			\node[fill=black] at (2,1) {};
			\node[fill=black] at (2,0) {};
			\node[fill=black] at (3,0) {};
			\node[fill=black] at (3,1) {};
			\node[fill=black] at (4,1) {};
			\draw (0,0) -- (1,0) -- (1,1) -- (2,1) -- (2,0) -- (3,0) -- (3,1) -- (4,1);
		\end{scope}
		\begin{scope}[ultra thick,inner sep =2pt,shape=circle,xshift=5cm]
			\node[fill=black] at (0,0) {};
			\node[fill=black] at (1,0) {};
			\node[fill=black] at (1,1) {};
			\node[fill=black] at (2,1) {};
			\node[fill=black] at (2,2) {};
			\node[fill=black] at (3,2) {};
			\node[fill=black] at (3,1) {};
			\node[fill=black] at (4,1) {};
			\draw (0,0) -- (1,0) -- (1,1) -- (2,1) -- (2,2) -- (3,2) -- (3,1) -- (4,1);
		\end{scope}
		\begin{scope}[ultra thick,inner sep =2pt,shape=circle,xshift=10cm,yshift=2cm]
			\node[fill=black] at (0,-2) {};
			\node[fill=black] at (1,-2) {};
			\node[fill=black] at (1,-1) {};
			\node[fill=black] at (2,-1) {};
			\node[fill=black] at (2,0) {};
			\node[fill=black] at (3,0) {};
			\node[fill=black] at (3,1) {};
			\node[fill=black] at (4,1) {};
			\draw (0,-2) -- (1,-2) -- (1,-1) -- (2,-1) -- (2,0) -- (3,0) -- (3,1) -- (4,1);
		\end{scope}
	\end{tikzpicture}
    \caption[{\it (i)}]{$(a,b)=(3,7)$}
    \label{fig:R<3/7}
    \end{subfigure}
    
    \caption{Examples of forbidden patterns for a radius $R<a/b$, see Lemma~\ref{lem:R<p/q}.}
    \label{fig:R<a/b}
\end{figure}

\subsection{Lower bounds}\label{sec:LBR_min}

We distinguish two separate cases for each model: $p=\infty$ and $p\in[1,\infty)$. For $p=\infty$, the proof consists in listing all possible beginnings of valid self-avoiding paths and showing that it is impossible to construct an infinite one, using Lemma~\ref{lem:R<p/q}. For $p\in[1,\infty)$, we assume that there exists an infinite self-avoiding path and show that a given positive quantity decreases at least linearly along this path, which leads to an absurdity.

\subsubsection{Complete neighbor model and $p=\infty$}\label{sec:LBR_minCNInfinity}

Let $u = (u_k)_{k \geq 0}$ be an infinite self-avoiding path of $\Gamma_{\infty,R}^{\cn}(P)$, for some $R<2/5$. By symmetry, we assume without loss of generality that $s_1\in\{\ra, \nea\}$.

To begin with, let us examine the case $s_1=\ra$. By Lemma~\ref{lem:R<p/q}(i), $s_2\notin\{\sea,\ra,\nea\}$. By symmetry, we assume that $s_2 \in \{\nwa,\ua\}$. 
Figure~\ref{fig:InfNorm} summarizes the case distinction that follows.

\begin{itemize}[label=$\bullet$]
\item Case $s_2=\ua$. By Lemma~\ref{lem:R<p/q}(i), $s_3\notin\{\nwa,\ua,\nea\}$ and since $u$ is self-avoiding, $s_3\notin\{\swa,\da\}$. So, $s_3\in\{\sea,\ra,\la\}$.
	\begin{itemize}[label=$\bullet$]
		\item Case $s_3=\sea$. By Lemma~\ref{lem:R<p/q}(i), $s_4\notin\{\swa,\da,\sea,\ra,\nea\}$ and since $u$ is self-avoiding, $s_4\notin\{\la,\nwa\}$. It follows that $s_4=\ua$. Let us now examine the possible values for $s_5$.
		\begin{itemize}[label=$\bullet$]
		    \item If $s_5\in\{\swa,\la,\da\}$, $u$ is not self-avoiding.
		    \item If $s_5\in\{\nwa,\ua,\nea\}$, by Lemma~\ref{lem:R<p/q}(i), $u$ is not valid.
		    \item If $s_5\in\{\sea,\ra\}$, by Lemma~\ref{lem:R<p/q}(ii), $u$ is not valid.
		\end{itemize}
So, there is no allowed value for $s_5$. 
Consequently, the pattern $(\ra,\ua,\sea)$ and all the patterns derived from it by symmetry or rotation (``green'' patterns, see Figure~\ref{fig:InfNorm}) are forbidden in $u$.

		For the rest of the proof, we make implicit the uses of the self-avoiding property and of Lemma~\ref{lem:R<p/q}(i). 

	\item Case $s_3=\ra$. Then, $s_4\in\{\da,\ua,\nwa\}$.
	\begin{itemize}[label=$\bullet$]
		\item If $s_4=\da$, then  by Lemma~\ref{lem:R<p/q}(ii), $u$ cannot be extended. 
		\item If $s_4=\ua$, then $s_5=\la$, and by Lemma~\ref{lem:R<p/q}(ii), $u$ cannot be extended. 
		\item If $s_4=\nwa$, then $u$ contains the green pattern $(\ua,\ra,\nwa)$, which is forbidden, see Line 4 of Figure~\ref{fig:InfNorm}. 
	\end{itemize}   
So, there is no allowed value for $s_4$.  Consequently, the pattern $(\ra,\ua,\ra)$ and all the patterns derived from it by symmetry or rotation (``blue'' patterns, see Figure~\ref{fig:InfNorm}) are forbidden in $u$.
	\item Case $s_3=\la$. Then $s_4\in\{\ua,\nea\}$.
		\begin{itemize}[label=$\bullet$]
		\item If $s_4=\ua$, then $u$ contains a blue pattern, which is forbidden,  see Line 5 of Figure~\ref{fig:InfNorm}.
		\item If $s_4=\nea$, then $u$ contains a green pattern, which is forbidden, see Line 6 of Figure~\ref{fig:InfNorm}.
		\end{itemize}
	So, there is no allowed value for $s_4$.
\end{itemize}

Consequently, if $s_1=\ra$ and $s_2=\ua$, there is no valid way to extend the path. It follows that the pattern $(\ra,\ua)$ and all the patterns derived from it by symmetry or rotation (``orange'' patterns, see Figure~\ref{fig:InfNorm}) are forbidden in $u$. 
\item Case $s_2 = \nwa$. Then $s_3=\ra$. We have $s_4\in\{\ua,\nwa\}$.
    \begin{itemize}[label=$\bullet$]
    	\item If $s_4=\ua$, then $u$ contains an orange pattern, which is forbidden, see Line 8 of Figure~\ref{fig:InfNorm}.
    	\item If $s_4=\nwa$, then $s_5=\ra$, and by Lemma~\ref{lem:R<p/q}(ii), $u$ cannot be extended.
    \end{itemize}
\end{itemize}

Finally, we obtain that $u$ cannot begin by $s_1=\ra$, and more generally, that an infinite self-avoiding path of $\Gamma_{\infty,R}^{\cn}(P)$ cannot contain any horizontal or vertical edges. But, by Lemma~\ref{lem:R<p/q}(i), having two consecutive diagonal edges is forbidden. So, we conclude that $u$ must be finite for any $R<2/5$.

\begin{figure}[ht!]
    \centering
    \scalebox{0.6}{
    \begin{tikzpicture}[scale=0.3]

    \draw [ultra thick] (0,-1) -- (2,-1);
    \draw  (3,-1) -- (5,11);
    \draw  (3,-1) -- (5,-13);
    
\node at (1,24) {$k=1$};

\draw [dashed] (-2,26) -- (-2,-23);
\draw [dashed] (4,26) -- (4,-23);

\begin{scope}[xshift=6cm,yshift=10cm]
    \draw [ultra thick,color=orange] (0,0) -- (2,0) -- (2,2);
    \draw  (3,1) -- (5,1);
    \draw  (3,1) -- (5,10);
    \draw  (3,1) -- (5,-11);
\end{scope}

\begin{scope}[xshift=6cm,yshift=-14cm]
    \draw [ultra thick] (0,0) -- (2,0) -- (0,2);
    \draw  (3,1) -- (5,1);
\end{scope}

\node at (7,24) {$k=2$};
\draw [dashed] (10,26) -- (10,-23);


\begin{scope}[xshift=12cm,yshift=19cm]
    \draw [ultra thick,color=green] (0,0) -- (2,0) -- (2,2) -- (4,0);
    \draw (5,1) -- (7,1);
\end{scope}

\begin{scope}[xshift=12cm,yshift=10cm]
    \draw [ultra thick,color=blue] (0,0) -- (2,0) -- (2,2) -- (4,2);
    \draw (5,1) -- (7,1); 
    \draw (5,1) -- (7,6); 
    \draw (5,1) -- (7,-5); 
\end{scope}

\begin{scope}[xshift=13cm,yshift=-2cm]
    \draw [ultra thick] (0,0) -- (2,0) -- (2,2) -- (0,2);
    \draw (4,1) -- (6,1);
    \draw (4,1) -- (6,-5);
\end{scope}

\begin{scope}[xshift=13cm,yshift=-14cm]
    \draw [ultra thick] (0,0) -- (2,0) -- (0,2) -- (2,2);
    \draw  (4,1) -- (6,1);
    \draw  (4,1) -- (6,-5);
\end{scope}

\node at (14,24) {$k=3$};
\draw [dashed] (18,26) -- (18,-23);


\begin{scope}[xshift=20cm,yshift=19cm]
    \draw [ultra thick] (0,0) -- (2,0) -- (2,2) -- (4,0) -- (4,2);
\end{scope}

\begin{scope}[xshift=20cm,yshift=15cm]
    \draw [ultra thick] (0,0) -- (2,0) -- (2,2) -- (4,2) -- (4,0);
\end{scope}

\begin{scope}[xshift=20cm,yshift=9cm]
    \draw [ultra thick] (0,0) -- (2,0) -- (2,2) -- (4,2) -- (4,4);
    \draw (5,2) -- (7,2);
\end{scope}

\begin{scope}[xshift=20cm,yshift=3cm]
    \draw [ultra thick] (0,0) -- (2,0);
    \draw [ultra thick,color=green] (2,-.06) -- (2,2) -- (4,2) -- (2,4);
\end{scope}

\begin{scope}[xshift=21cm,yshift=-3cm]
    \draw [ultra thick] (0,0) -- (2,0);
    \draw [ultra thick,color=blue] (2,-.06) -- (2,2) -- (0,2) -- (0,4);
\end{scope}

\begin{scope}[xshift=21cm,yshift=-9cm]
    \draw [ultra thick] (0,0) -- (2,0);
    \draw [ultra thick,color=green] (2,-.06) -- (2,2) -- (0,2) -- (2,4);
\end{scope}

\begin{scope}[xshift=21cm,yshift=-15cm]
    \draw [ultra thick] (0,0) -- (2,0) -- (0,2) -- (2,2) -- (0,4);
    \draw  (4,2) -- (6,2);
\end{scope}

\begin{scope}[xshift=21cm,yshift=-21cm]
    \draw [ultra thick] (0,0) -- (2,0) -- (0,2);
    \draw [ultra thick,color=orange] (-.05,2) -- (2,2) -- (2,4);
\end{scope}

\node at (22,24) {$k=4$};
\draw [dashed] (26,26) -- (26,-23);


\begin{scope}[xshift=28cm,yshift=9cm]
    \draw [ultra thick] (0,0) -- (2,0);
    \draw [ultra thick] (2,-.06) -- (2,2) -- (4,2) -- (4,4) -- (2,4);
\end{scope}

\begin{scope}[xshift=29cm,yshift=-15cm]
    \draw [ultra thick] (0,0) -- (2,0) -- (0,2) -- (2,2) -- (0,4) -- (2,4);
\end{scope}

\node at (30,24) {$k=5$};
\draw [dashed] (34,26) -- (34,-23);

\node at (38,20) {\gras{Line 1}};
\node at (38,16) {\gras{Line 2}};
\node at (38,11) {\gras{Line 3}};
\node at (38,5) {\gras{Line 4}};
\node at (38,-1) {\gras{Line 5}};
\node at (38,-7) {\gras{Line 6}};
\node at (38,-13) {\gras{Line 7}};
\node at (38,-19) {\gras{Line 8}};

\end{tikzpicture}
    }
    \caption{Possible self-avoiding paths of length $k$ of $\Gamma_{\infty,R}^{\cn}(P)$ with starting step $s_1=\ra$, for $R<2/5$.
    }
    \label{fig:InfNorm}
\end{figure}
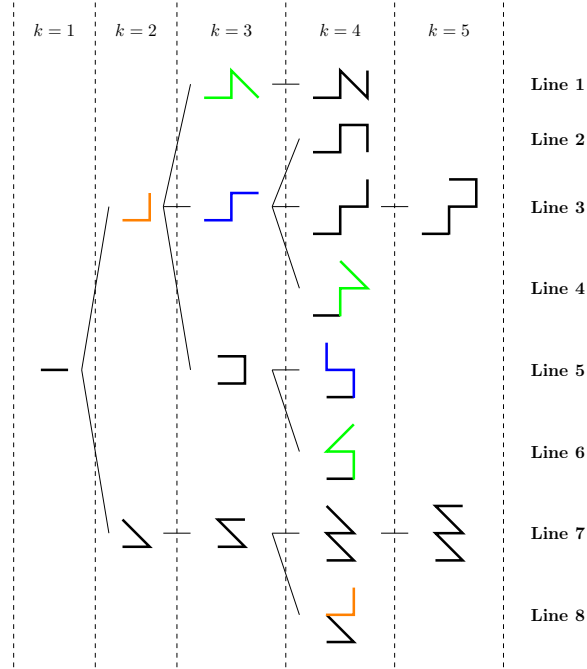

\subsubsection{Nearest neighbor model and $p=\infty$}\label{sec:LBR_minNNInfinity}

Let $u = (u_k)_{k \geq 0}$ be an infinite self-avoiding path of $\Gamma_{\infty,R}^{\nn}(P)$, for some $R<3/7$. By definition of the model, $s_k\in\{\la,\da,\ua,\ra\}$ for all $k\geq 0$. Furthermore, by Lemma~\ref{lem:R<p/q}(i), horizontal and vertical steps alternate in $u$. We assume without loss of generality that $(s_1,s_2)=(\ra,\ua)$, which implies that for all $k\geq0$, $s_{2k+1}\in\{\la,\ra\}$ and $s_{2k+2}\in\{\da,\ua\}$. For the rest of the proof, we make implicit the use of the self-avoiding property and of Lemma~\ref{lem:R<p/q}(i). Figures~\ref{fig:RminInfP1} and~\ref{fig:RminInfP2} summarize the case distinction that follows.

\begin{itemize}[label=$\bullet$]
    \item Case $(s_3,s_4)=(\ra,\ua)$.
    \begin{itemize}[label=$\bullet$]    
        \item Case $(s_5,s_6)=(\ra,\ua)$. By Lemma~\ref{lem:R<p/q}(iii), we must then have $s_7=\la$, see Line~1 of Figure~\ref{fig:RminInfP2}. By applying Lemma~\ref{lem:R<p/q}(iii) once again, we obtain that $u$ cannot be extended.
        \item Case $(s_5,s_6)=(\ra,\da)$. By Lemma~\ref{lem:R<p/q}(iii), $u$ cannot be extended, see Line~2 of Figure~\ref{fig:RminInfP2}.
        \item Case $s_5=\la$.
        For the study of this case,
        we refer to Figure~\ref{fig:RminInfP2} (Lines 3 and 4), that examines the different ways to continue the path, and shows that they all eventually reach a stage where they can no longer be extended.
    \end{itemize}
    Consequently, the pattern $(\ra,\ua,\ra,\ua)$ and all the patterns derived from it by symmetry or rotation (``red'' patterns, see Figure~\ref{fig:RminInfP1} and~\ref{fig:RminInfP2}) are forbidden in $u$.

    \item Case $(s_3,s_4)=(\ra,\da)$. We have $s_5=\ra$.
    If $s_6=\da$, then $u$ contains a red pattern, which is forbidden, see Line~2 of Figure~\ref{fig:RminInfP1}. So, $s_6=\ua$. By Lemma~\ref{lem:R<p/q}(iii), $u$ cannot be extended, see Line~3 of Figure~\ref{fig:RminInfP1}.
\end{itemize}
    Consequently, the pattern $(\ra,\ua,\ra)$ and all the patterns derived from it by symmetry or rotation (``blue'' patterns, see Figure~\ref{fig:RminInfP1}) are forbidden in $u$.   
\begin{itemize}
    \item Case $s_3=\la$. Then, $s_4=\ua$. But $u$ contains a blue pattern, which is forbidden, see Line~4 of Figure~\ref{fig:RminInfP1}.
\end{itemize}

We conclude that for $R<3/7$, $u$ must be finite which implies that $R_{\min}\geq 3/7$.

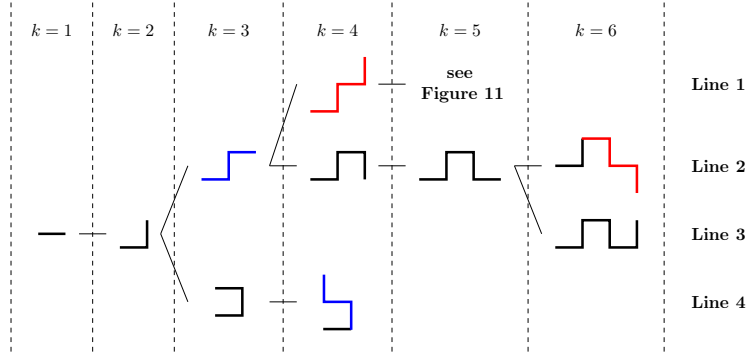
\begin{figure}[ht!]
    \centering
    \scalebox{0.6}{
    \begin{tikzpicture}[scale=0.3]

    \draw [ultra thick] (0,0) -- (2,0);
    \draw  (3,0) -- (5,0);

\node at (1,15) {$k=1$};
\draw [dashed] (-2,17) -- (-2,-9);
\draw [dashed] (4,17) -- (4,-9);

\begin{scope}[xshift=6cm,yshift=-1cm]
    \draw [ultra thick] (0,0) -- (2,0) -- (2,2);
    \draw (3,1) -- (5,6);
    \draw (3,1) -- (5,-4);
\end{scope}

\node at (7,15) {$k=2$};
\draw [dashed] (10,17) -- (10,-9);

\begin{scope}[xshift=12cm,yshift=4cm]
    \draw [ultra thick,color=blue] (0,0) -- (2,0) -- (2,2) -- (4,2);
    \draw (5,1) -- (7,7); 
    \draw (5,1) -- (7,1); 
\end{scope}

\begin{scope}[xshift=13cm,yshift=-6cm]
    \draw [ultra thick] (0,0) -- (2,0) -- (2,2) -- (0,2);
    \draw (4,1) -- (6,1);
\end{scope}

\node at (14,15) {$k=3$};
\draw [dashed] (18,17) -- (18,-9);


\begin{scope}[xshift=20cm,yshift=9cm]
    \draw [ultra thick,color=red] (0,0) -- (2,0) -- (2,2) -- (4,2) -- (4,4);
    \draw (5,2) -- (7,2);
    \node at (11,2.75) {\gras{see}};
    \node at (11.2,1.25) {\gras{Figure~\ref{fig:RminInfP2}}};
\end{scope}

\begin{scope}[xshift=20cm,yshift=4cm]
    \draw [ultra thick] (0,0) -- (2,0) -- (2,2) -- (4,2) -- (4,0);
    \draw (5,1) -- (7,1);
\end{scope}

\begin{scope}[xshift=21cm,yshift=-7cm]
    \draw [ultra thick] (-.054,0) -- (2,0);
    \draw [ultra thick,color=blue] (2,-.054) -- (2,2) -- (0,2) -- (0,4);
\end{scope}

\node at (22,15) {$k=4$};

\draw [dashed] (26,17) -- (26,-9);


\begin{scope}[xshift=28cm,yshift=4cm]
    \draw [ultra thick] (0,0) -- (2,0) -- (2,2) -- (4,2) -- (4,0) -- (6,0);
    \draw (7,1) -- (9,1);
    \draw (7,1) -- (9,-4);
\end{scope}

\node at (31,15) {$k=5$};
\draw [dashed] (36,17) -- (36,-9);


\begin{scope}[xshift=38cm,yshift=-1cm]
    \draw [ultra thick] (0,0) -- (2,0) -- (2,2) -- (4,2) -- (4,0) -- (6,0) -- (6,2);
\end{scope}

\begin{scope}[xshift=38cm,yshift=5cm]
    \draw [ultra thick] (0,0) -- (2,0) -- (2,2);
    \draw [ultra thick,color=red] (1.946,2) -- (4,2) -- (4,0) -- (6,0) -- (6,-2);
\end{scope}

\node at (41,15) {$k=6$};
\draw [dashed] (46,17) -- (46,-9);

\node at (50,11) {\gras{Line 1}};
\node at (50,5) {\gras{Line 2}};
\node at (50,0) {\gras{Line 3}};
\node at (50,-5) {\gras{Line 4}};
\end{tikzpicture}
    }
    \caption{Possible self-avoiding paths of length $k$ of $\Gamma_{\infty,R}^{\nn}(P)$ for $R<3/7$ using Lemma~\ref{lem:R<p/q}.}
    \label{fig:RminInfP1}
\end{figure}
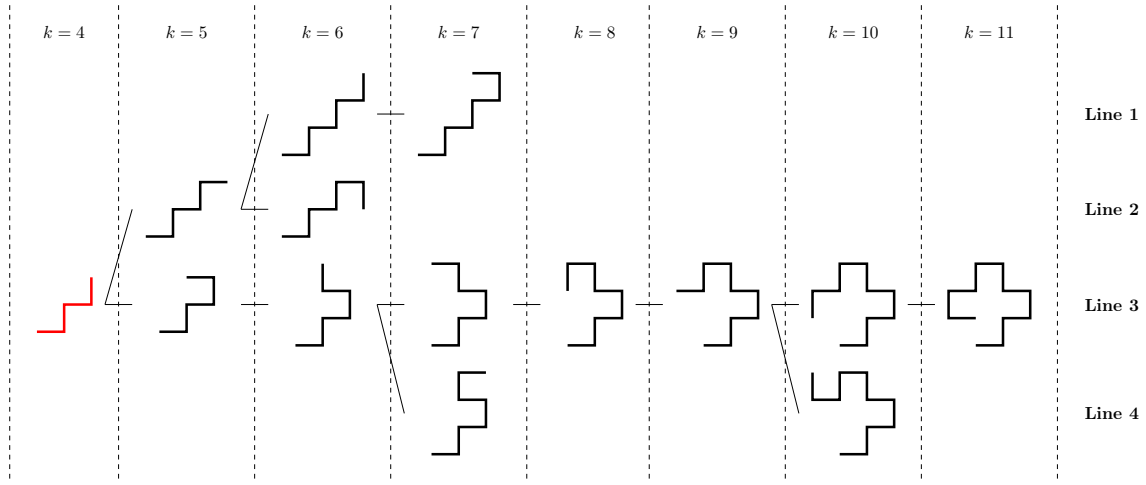
\begin{figure}[ht!]
    \centering
    \scalebox{0.6}{
    \begin{tikzpicture}[scale=0.3]


\draw [ultra thick,color=red] (0,0) -- (2,0) -- (2,2) -- (4,2) -- (4,4);
\draw (5,2) -- (7,9);
\draw (5,2) -- (7,2);

\node at (2,22) {$k=4$};
\draw [dashed] (-2,24) -- (-2,-11);
\draw [dashed] (6,24) -- (6,-11);


\begin{scope}[xshift=8cm,yshift=7cm]
       \draw [ultra thick,color=black] (0,0) -- (2,0) -- (2,2) -- (4,2) -- (4,4) -- (6,4);
       \draw (7,2) -- (9,2);
       \draw (7,2) -- (9,9);
\end{scope}

\begin{scope}[xshift=9cm,yshift=0cm]
       \draw [ultra thick,color=black] (0,0) -- (2,0) -- (2,2) -- (4,2) -- (4,4) -- (2,4);
       \draw (6,2) -- (8,2);
\end{scope}

\node at (11,22) {$k=5$};
\draw [dashed] (16,24) -- (16,-11);


\begin{scope}[xshift=18cm,yshift=13cm]
       \draw [ultra thick,color=black] (0,0) -- (2,0) -- (2,2) -- (4,2) -- (4,4) -- (6,4) -- (6,6);
       \draw (7,3) -- (9,3);
\end{scope}

\begin{scope}[xshift=18cm,yshift=7cm]
       \draw [ultra thick,color=black] (0,0) -- (2,0) -- (2,2) -- (4,2) -- (4,4) -- (6,4) -- (6,2);
\end{scope}

\begin{scope}[xshift=19cm,yshift=-1cm]
       \draw [ultra thick,color=black] (0,0) -- (2,0) -- (2,2) -- (4,2) -- (4,4) -- (2,4) -- (2,6);
       \draw (6,3) -- (8,3);
       \draw (6,3) -- (8,-5);
\end{scope}

\node at (21,22) {$k=6$};
\draw [dashed] (26,24) -- (26,-11);

\begin{scope}[xshift=28cm,yshift=13cm]
       \draw [ultra thick,color=black] (0,0) -- (2,0) -- (2,2) -- (4,2) -- (4,4) -- (6,4) -- (6,6) -- (4,6);
\end{scope}

\begin{scope}[xshift=29cm,yshift=-1cm]
       \draw [ultra thick,color=black] (0,0) -- (2,0) -- (2,2) -- (4,2) -- (4,4) -- (2,4) -- (2,6) -- (0,6);
       \draw (6,3) -- (8,3);
\end{scope}

\begin{scope}[xshift=29cm,yshift=-9cm]
       \draw [ultra thick,color=black] (0,0) -- (2,0) -- (2,2) -- (4,2) -- (4,4) -- (2,4) -- (2,6) -- (4,6);
\end{scope}

\node at (31,22) {$k=7$};
\draw [dashed] (36,24) -- (36,-11);


\begin{scope}[xshift=39cm,yshift=-1cm]
       \draw [ultra thick,color=black] (0,0) -- (2,0) -- (2,2) -- (4,2) -- (4,4) -- (2,4) -- (2,6) -- (0,6) -- (0,4);
       \draw (5,3) -- (7,3);
       
\end{scope}

\node at (41,22) {$k=8$};
\draw [dashed] (45,24) -- (45,-11);

\begin{scope}[xshift=49cm,yshift=-1cm]
       \draw [ultra thick,color=black] (0,0) -- (2,0) -- (2,2) -- (4,2) -- (4,4) -- (2,4) -- (2,6) -- (0,6) -- (0,4) --(-2,4);
       \draw (5,3) -- (7,3);
       \draw (5,3) -- (7,-5);
\end{scope}

\node at (50,22) {$k=9$};
\draw [dashed] (55,24) -- (55,-11);

\begin{scope}[xshift=59cm,yshift=-1cm]
       \draw [ultra thick,color=black] (0,0) -- (2,0) -- (2,2) -- (4,2) -- (4,4) -- (2,4) -- (2,6) -- (0,6) -- (0,4) -- (-2,4) -- (-2,2);
       \draw (5,3) -- (7,3);
\end{scope}

\begin{scope}[xshift=59cm,yshift=-9cm]
       \draw [ultra thick,color=black] (0,0) -- (2,0) -- (2,2) -- (4,2) -- (4,4) -- (2,4) -- (2,6) -- (0,6) -- (0,4) -- (-2,4) -- (-2,6);
\end{scope}

\node at (60,22) {$k=10$};
\draw [dashed] (65,24) -- (65,-11);

\begin{scope}[xshift=69cm,yshift=-1cm]
       \draw [ultra thick,color=black] (0,0) -- (2,0) -- (2,2) -- (4,2) -- (4,4) -- (2,4) -- (2,6) -- (0,6) -- (0,4) -- (-2,4) -- (-2,2) -- (0,2);
\end{scope}

\node at (70,22) {$k=11$};
\draw [dashed] (75,24) -- (75,-11);

\node at (79,16) {\gras{Line 1}};
\node at (79,9) {\gras{Line 2}};
\node at (79,2) {\gras{Line 3}};
\node at (79,-6) {\gras{Line 4}};

\end{tikzpicture}
    }
    \caption{Possible self-avoiding paths of length $k$ of $\Gamma_{\infty,R}^{\nn}(P)$ for $R<3/7$ when $(s_1,s_2,s_3,s_4)=(\ra,\ua,\ra,\ua)$.}
    \label{fig:RminInfP2}
\end{figure}
\subsubsection{Complete neighbor model and $p\in[1,\infty)$}\label{sec:LBR_MinCN}
For $p\in[1,\infty)$, the proof requires more subtle arguments: let $R < \displaystyle\min \left\{\frac{1}{2},\frac{2}{1+2^{2-\frac{1}{p}}}\right\}$ and let $u=(u_k)_{k\geq0}$ be an infinite self-avoiding path of $\Gamma_{p,R}^{\cn}(P)$ for some $P\in\mathcal{P}$. We construct a subsequence $(u_{a_k})_{k\geq1}$, and we study the distance $A_k$ from the point $P_{u_{a_k}}$ to a certain border of the cell to which it belongs. We prove that $(A_k)_{k\geq0}$ decreases at least linearly and becomes negative, which leads to a contradiction. This implies that the sequence $u$ must be finite.

\paragraph{Notations.}
Let us denote by $\Diag=\{\swa,\nwa,\sea,\nea\}$ the set of diagonal edges. We suppose that $u$ is a self-avoiding path of $\Gamma_{p,R}^{\cn}(P)$, and we define the sequence $(a_k)_{k\geq1}$ by 
\begin{itemize}
	\item $a_1 = \left \{
	\begin{array}{ll}
		1 \mbox{ if } s_2\notin \Diag,\\
		2 \mbox{ otherwise.} 
	\end{array}
	\right.$\\
	\item for $k\geq1$, $a_{k+1} = \left \{
	\begin{array}{ll}
		a_k + 2 \mbox{ if } \|u_{a_k+5}-u_{a_k}\|_{\infty} \neq 3,\\
		a_k + 3 \mbox{ otherwise.} 
	\end{array}
	\right.
	$
\end{itemize}

From the sequence $(a_k)_{k \geq 1}$, we define, for any $k \geq 1$, 
$$\delta_{k} = \left \{
	\begin{array}{ll}
		s_{a_{k} +2} &\mbox{ if } s_{a_{k} +2}\notin \Diag,\\						
		s_{a_{k} +1} + s_{a_{k} +2} &\mbox{ otherwise.} 
	\end{array}
	\right.
$$

One can prove, using Lemma~\ref{lem:R<p/q}(i), that, for any $k \geq 1$, $\delta_k \in \{\la,\ra,\da,\ua\}$.

Moreover, for any $k \geq 1$, we define
\begin{equation} \label{eq:Ak}
A_k = \left\{ 
	\begin{array}{ll}
		X_{u_{a_k}}-x_{u_{a_k}} &\mbox { if } \delta_k = \ra,\\
		1+x_{u_{a_k}}-X_{u_{a_k}} &\mbox { if } \delta_k = \la,\\
		Y_{u_{a_k}}-y_{u_{a_k}} &\mbox { if } \delta_k = \ua,\\
		1+y_{u_{a_k}}-Y_{u_{a_k}} &\mbox { if } \delta_k = \da.
	\end{array}
	\right.
\end{equation}

The quantity $A_k$ represents the distance between $P_{u_{a_k}}$ and the border of its cell specified by $\delta_k$. More precisely, if $\delta_k=\ra$, $A_k$ is the distance between $P_{u_{a_k}}$ and the left border of its cell, while if $\delta_k=\da$, it is the distance with the top border, and similarly for the other possible directions, see Figure~\ref{fig:a_k&delta_k} for an illustration.\par

\medskip

Our aim is to control the increment $A_{k+1}-A_k$. To this end, we first look at the possible patterns we can observe between the points $u_{a_k}$ and $u_{a_{k+1}}$.

\begin{figure}[ht!]
    \begin{minipage}{.4\textwidth}
    \centering
    \begin{tikzpicture}
	\node[fill=black,inner sep=2pt,shape=circle] at (1,0){};
	\draw (0,0) -- (1,0);
	\node[fill=black,inner sep=2pt,shape=circle] at (1.0,1.0){};
	\draw (1,0) -- (1.0,1.0);
	\node[fill=black,inner sep=2pt,shape=circle] at (2.0,0.0){};
	\draw (1.0,1.0) -- (2.0,0.0);
	\node[fill=black,inner sep=2pt,shape=circle] at (2.0,1.0){};
	\draw (2.0,0.0) -- (2.0,1.0);
	\node[fill=black,inner sep=2pt,shape=circle] at (3.0,0.0){};
	\draw (2.0,1.0) -- (3.0,0.0);
	\node[fill=black,inner sep=2pt,shape=circle] at (3,1){};
	\draw (3.0,0.0) -- (3.0,1.0);
	\node[fill=black,inner sep=2pt,shape=circle] at (4,1){};
	\draw (3.0,1.0) -- (4.0,1.0);
	\node[fill=black,inner sep=2pt,shape=circle] at (3,2){};
	\draw (4.0,1.0) -- (3.0,2.0);
	\node[fill=black,inner sep=2pt,shape=circle] at (4,2){};
	\draw (3.0,2.0) -- (4.0,2.0);
	\node[fill=black,inner sep=2pt,shape=circle] at (4,3){};
	\draw (4.0,2.0) -- (4.0,3.0);
	\node[fill=black,inner sep=2pt,shape=circle] at (5,3){};
	\draw (4.0,3.0) -- (5.0,3.0);
	\node[fill=black,inner sep=2pt,shape=circle] at (5,4){};
	\draw (5.0,3.0) -- (5.0,4.0);
	\node[fill=black,inner sep=2pt,shape=circle] at (6,3){};
	\draw (5.0,4.0) -- (6.0,3.0);
	\node[fill=black,inner sep=2pt,shape=circle] at (6,4){};
	\draw (6.0,3.0) -- (6.0,4.0);
	\node[fill=black,inner sep=2pt,shape=circle] at (7,4){};
	\draw (6.0,4.0) -- (7.0,4.0);
	\node[fill=black,inner sep=2pt,shape=circle] at (6,5){};
	\draw (7.0,4.0) -- (6.0,5.0);
	\node[draw=black,fill=black,inner sep=2pt,shape=circle,label={below :$u_0$}] at (0,0){};
	\node [draw=black,fill=red,inner sep=2pt,shape=circle,label={below :{\color{red}$\ra$}}] at (1,0) {};
	\node [draw=black,fill=red,inner sep=2pt,shape=circle,label={below :{\color{red}$\ra$}}] at (2,0){};
	\node [draw=black,fill=red,inner sep=2pt,shape=circle,label={below :{\color{red}$\ra$}}] at (3,0){};
	\node [draw=black,fill=red,inner sep=2pt,shape=circle,label={left :{\color{red}$\ua$}}] at (3,2){};
	\node [draw=black,fill=red,inner sep=2pt,shape=circle,label={below :{\color{red}$\ra$}}] at (5,3){};
	\node [draw=black,fill=red,inner sep=2pt,shape=circle,label={below :{\color{red}$\ra$}}] at (6,3){};
    \end{tikzpicture}
    \end{minipage}
    \hspace{1cm}
    \begin{minipage}{.4\textwidth}
    \scalebox{1.1}{
    \begin{tikzpicture}[>={Latex[length=2mm]}]
	\draw[step=1cm,color=gray,ultra thin] (-.4,-.4) grid (6.4,4.4);
	\draw [pattern color=orange,pattern=north east lines] (1.47,1) rectangle (2,0);
	\draw[-Classical TikZ Rightarrow,color=red] (1,.5) -- (1.47,.5);
	\draw [pattern color=orange,pattern=north east lines] (2.44,1) rectangle (3,0);
	\draw[-Classical TikZ Rightarrow,color=red] (2,.5) -- (2.44,.5);
	\draw [pattern color=orange,pattern=north east lines] (3.4,1) rectangle (4,0);
	\draw[-Classical TikZ Rightarrow,color=red] (3,.5) -- (3.4,.5);
	\draw [pattern color=orange,pattern=north east lines] (4,2.229) rectangle (3,3);
	\draw[-Classical TikZ Rightarrow,color=red] (3.5,2) -- (3.5,2.229);
	\draw [pattern color=orange,pattern=north east lines] (5.0519,3) rectangle (6,4);
	\draw[-Classical TikZ Rightarrow,color=red] (5,3.5) -- (5.0519,3.5);
	\draw[->] (1,1) -- (0.7,0.6);
	\node[shape=circle,inner sep=2pt,fill=black,draw=black] (A) at (1,1) {};
	\draw[->] (1.47,1) -- (1.47,0.6);
	\node[shape=circle,inner sep=2pt,fill=red,draw=black] (B) at (1.48,1) {};
	\draw[->] (1.96,1) -- (1.96,1.4);	
	\node[shape=circle,inner sep=2pt,fill=black] (C) at (1.96,1) {};
	\draw[->] (2.44,1) -- (2.44,0.6);
	\node[shape=circle,inner sep=2pt,fill=red,draw=black] (D) at (2.44,1) {};
	\draw[->] (2.92,1) -- (2.92,1.4);
	\node[shape=circle,inner sep=2pt,fill=black] (E) at (2.92,1) {};
	\draw[->] (3.4,1) -- (3.4,.6);
	\node[shape=circle,inner sep=2pt,fill=red,draw=black] (F) at (3.4,1) {};
	
	\node[shape=circle,inner sep=2pt,fill=black] (G) at (3.7,1.374) {};
	
	\draw[->] (4,1.749) -- (4.4,1.749);
	\node[shape=circle,inner sep=2pt,fill=black,draw=black] (H) at (4,1.749) {};
	\draw[->] (4,2.229) -- (3.6,2.229);
	\node[shape=circle,inner sep=2pt,fill=red,draw=black] (I) at (4,2.229) {};
	
	\node[shape=circle,inner sep=2pt,fill=black,draw=black] (J) at (4.2859,2.6145) {};
	
	\draw[->] (4.5719,3) -- (4.5719,3.4);
	\node[shape=circle,inner sep=2pt,fill=black,draw=black] (K) at (4.5719,3) {};
	\draw[->] (5.0519,3) -- (5.0519,3.4);
	\node[shape=circle,inner sep=2pt,fill=red,draw=black] (L) at (5.0519,3) {};
	\draw (A) -- (B) -- (C) -- (D) -- (E) -- (F) -- (G) -- (H) -- (I) -- (J) -- (K) -- (L);
	
    \end{tikzpicture}
    }
    \end{minipage}    
    
    \caption{On the left, a self-avoiding path in $\ZZ^2$ with steps in $\step$. The red dot represents $u_0$, and the blue dots the sequence $(u_{a_k})_{k\geq1}$. Next to these points are the arrows given by the sequence $(\delta_k)_{k\geq1}$. On the right, we construct for $(R,p)=(0.48,2)$ a configuration of points $P$ realising the beginning of $u$, for which each $A_k$ is maximized, represented by the red arrows. It appears that it is not possible to place $P_{u_{a_6}}$. The orange hatched rectangles represent the forbidden domains for $(P_{u_{a_k}})_{k\geq1}$.}
    \label{fig:a_k&delta_k}
\end{figure}
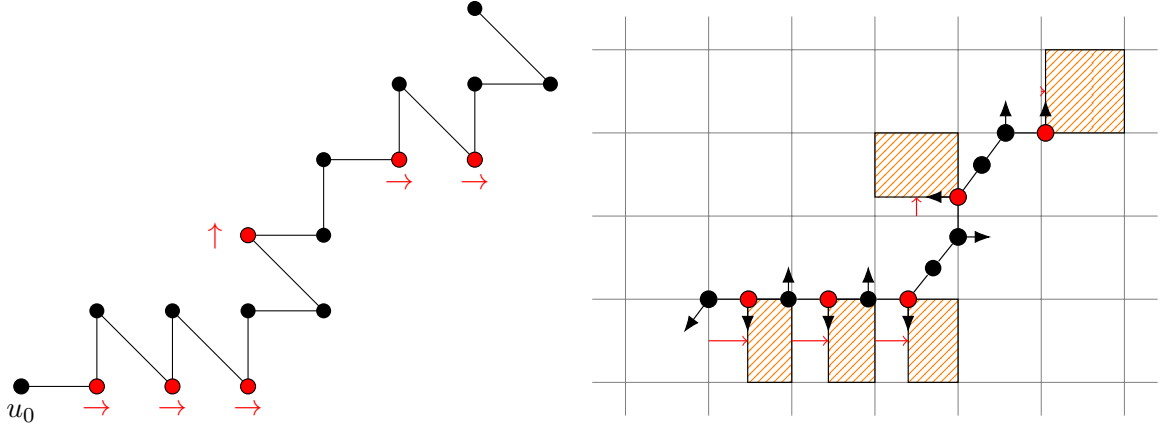

\paragraph{List of patterns between $u_{a_k}$ and $u_{a_{k+1}}$.}

Let $k\geq1$.
We identify all the possible patterns between $u_{a_k}$ and $u_{a_{k+1}}$, and control the value of $A_{k+1}-A_k$ accordingly. 

In all that follows, we use implicitly the self-avoiding property and Lemma~\ref{lem:R<p/q}(i).

Without loss of generality, we only consider the two cases $s_{a_{k}+1}=\ra$ and $s_{a_{k}+1} = \nea$. Let us first treat the case $s_{a_{k}+1} = \ra$. By symmetry, we can assume that $s_{a_k}\in\{\nwa,\ua\}$. 
If $s_{a_k}=\ua$, we have $s_{a_k+2}\in\{\nwa,\da,\ua\}$, and if $s_{a_k}=\nwa$, then $s_{a_k+2}\in\{\swa,\nwa,\ua\}$.

Figure~\ref{fig:PatternP} displays all of these possible scenarios. In the following, we use a letter together with one or several numbers to refer to the corresponding cells on Figure~\ref{fig:PatternP}.

Le us start by examining the possibilities described on the left table of Figure~\ref{fig:PatternP}, which correspond to the case where $s_{a_k+2}\in\{\nwa,\ua\}$.

%

\begin{itemize}[label=$\bullet$]
	\item \gras{(A.2,3,4)} Case $s_{a_k+2} = \nwa$. Then $s_{a_k+3}=\ra$. Since we have $\|u_{a_{k}+3}-u_{a_k}\|_{\infty} = 1$ and $\|u_{l+2}-u_{l}\|_{\infty}=1$ for any $l\geq1$, we obtain that $\|u_{a_{k}+5}-u_{a_k}\|_{\infty}\leq 2$. So $a_{k+1}=a_k+2$.
	\item \gras{(B.2)} Case $s_{a_k+2} = \ua$. Then $s_{a_k+3}\in\{\la,\sea,\ra\}$.
	\begin{itemize}[label=$\bullet$]
		\item \gras{(B.3,4)} If $s_{a_k+3}=\la$, we have $\|u_{a_{k}+3}-u_{a_k}\|_{\infty} = 1$. So $a_{k+1}=a_k+2$.
		\item \gras{(C.3,4,5)} If $s_{a_k+3}=\sea$, then $s_{a_k+4}=\ua$. It implies that $s_{a_k+5}\in\{\sea,\ra\}$. So $\|s_{a_{k}+5}-s_{a_k}\|_\infty=3$ and $a_{k+1}=a_k+3$.
		\item \gras{(D.3)} If $s_{a_k+3}=\ra$, then $s_{a_k+4}\in\{\nwa,\da,\ua\}$.
		\begin{itemize}[label=$\bullet$]
			\item \gras{(D.4,5)} If $s_{a_k+4} = \nwa$, then $s_{a_k+5} = \ra$. We have $\|s_{a_{k}+5}-s_{a_k}\|_\infty=2$, so $a_{k+1}=a_k+2$.
			\item \gras{(E.4)} If $s_{a_k+4} = \ua$, then $s_{a_k+5} \in\{\la,\sea,\ra\}$.
			\begin{itemize}[label=$\bullet$]
				\item \gras{(E.5)} If $s_{a_k+5}\in\{\sea,\ra\}$, then $\|s_{a_{k}+5}-s_{a_k}\|_{\infty}=3$, so we have $a_{k+1}=a_k+3$.
				\item \gras{(F.5)} If $s_{a_k+5}=\la$, then $\|s_{a_{k}+5}-s_{a_k}\|_{\infty}=2$, so we have $a_{k+1}=a_k+2$.
			\end{itemize}
			\item \gras{(G.4,5)} If $s_{a_k+4} = \da$, then $s_{a_k+5} \in\{\ra,\nea\}$ and we have $\|s_{a_{k}+5}-s_{a_k}\|_{\infty}=3$, so we have $a_{k+1}=a_k+3$.
			\end{itemize}
		\end{itemize}
\end{itemize}

For the study of the case $(s_{a_k},s_{a_k+2})\in\left\{\left(\nwa,\swa\right),\left(\ua,\da\right)\right\}$, we refer to the right table of Figure~\ref{fig:PatternP}. 
Let us only mention that if $(s_{a_k},s_{a_k+1},s_{a_k+2}) = (\nwa,\ra,\swa)$, which corresponds to cell (A'.2), the path cannot be extended, so that we can exclude this case.

\begin{figure}[ht!]
	\centering
	\scalebox{.3}{%
	\input{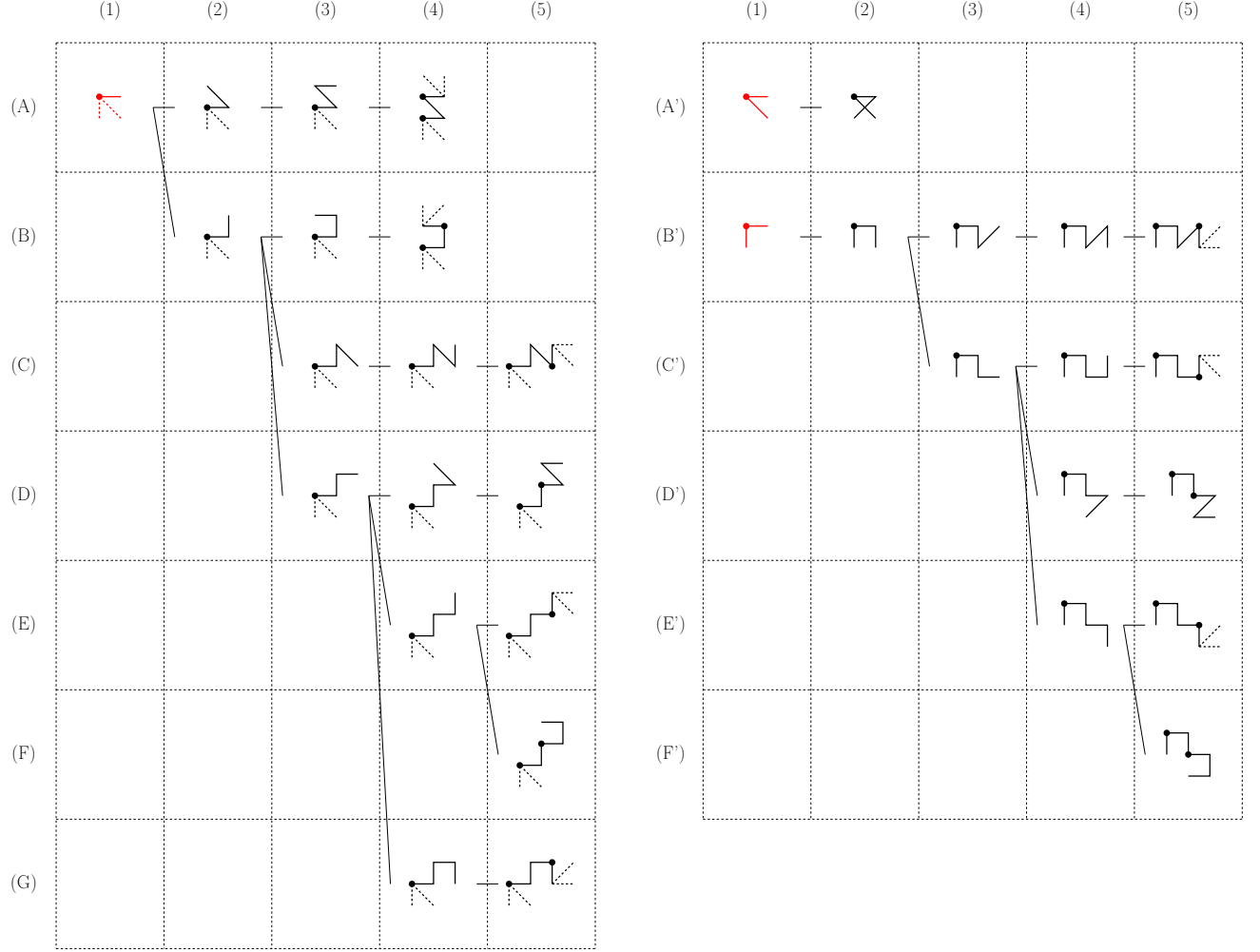}
	}
	\caption{List of possible patterns between the points $u_{a_k}$ and $u_{a_{k+1}}$, which are represented by the nodes.}
	\label{fig:PatternP}
\end{figure}

The above analysis shows that if $s_{a_k+1}\notin \Diag$ then $s_{a_{k+1}+1}\notin \Diag$. Furthermore, by Lemma~\ref{lem:R<p/q}(i), if $s_2 \in \Diag$, then $s_3 \notin \Diag$. Therefore, in all cases, we have $s_{{a_1}+1}\notin \Diag$.  By induction, it follows that for all $k\geq1$, $s_{a_{k}+1}\notin \Diag$. Hence, we can exclude the case $s_{a_{k}+1}= \nea$.

Furthermore, the observation of the two tables of Figure~\ref{fig:PatternP} shows that, whatever the pattern to which $s_{a_k}$ belongs, the evolution from $s_{a_{k+1}}$ is then given by the left table (see the possible steps after $s_{a_{k+1}}$ in columns $(4)$ and $(5)$ of the two tables). By induction, this means that for $k\geq2$, up to a rotation or symmetry, we always have \begin{equation}\label{eq:uShape}
    (s_{a_{k+1}},s_{a_{k+1}+1},s_{a_{k+1}+2}) \in\{\nwa,\ua\}\times\{\ra\}\times\{\nwa,\ua\}. 
\end{equation}

Moreover, for $k\geq2$,
\begin{itemize}
    \item if $a_{k+1} - a_k = 2$, then $y_{u_{a_{k+1}}}-y_{u_{a_k}}=1$ and $\delta_{k+1}=\delta_k=\ua$, see cells (A.4), (B.4), (D.5), (F.5) of Figure~\ref{fig:PatternP};
    \item if $a_{k+1} - a_k = 3$, then $s_{a_k+2} = \ua$,  $s_{a_{k}+3}\in\{\nwa,\ua\}$ and $\delta_{k+1}=\ra$, see cells (C.5), (E.5), (G.5) of Figure~\ref{fig:PatternP}.
\end{itemize}

To conclude the proof, let us now show that there exists a constant $c>0$ such that for all $k\geq2$, we have
$$A_{k+1} - A_k\leq-\min\{1-2R,c\}<0.$$

For that, we consider separately the cases $a_{k+1}-a_k=2$ and $a_{k+1}-a_k=3$. As before, we assume without loss of generality that equation \eqref{eq:uShape} is satisfied.

\paragraph{Case \boldmath$a_{k+1}-a_k=2$.} From the above, we have $y_{u_{a_{k+1}}}-y_{u_{a_k}}=1$ and $\delta_{k+1}=\delta_k=\ua$. Thus,  
$$A_{k+1} - A_k = (Y_{u_{a_{k+1}}}-y_{u_{a_{k+1}}}) - (Y_{u_{a_k}}-y_{u_{a_k}}) = Y_{u_{a_{k+1}}} - Y_{u_{a_k}} -1.$$
Since $a_{k+1}=a_k+2$, it follows that $Y_{u_{a_{k+1}}}\leq  Y_{u_{a_k}} + 2R$. So,
$$A_{k+1} - A_k \leq -(1-2R)<0.$$

\paragraph{Case \boldmath$a_{k+1}-a_k = 3$.}
As in Section~\ref{sec:UBR_min}, we set  $\displaystyle f(p) = \frac{2}{1+2^{2-\frac{1}{p}}}$.\vspace{0.5cm}

From the above, we have   $(s_{a_k},s_{a_k+1},s_{a_k+2},s_{a_k+3})\in\{\nwa,\ua\}\times\{\ra\}\times\{\ua\}\times\{\sea,\ra\}$, see Figure~\ref{fig:CN3}. 

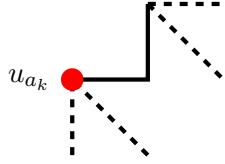
\begin{figure}[ht!]
    \centering
    \begin{tikzpicture}
        \draw[ultra thick] (0,0) -- (1,0) -- (1,1);
        \draw[ultra thick, dashed] (0,-1) -- (0,0);
        \draw[ultra thick, dashed] (1,-1) -- (0,0);
        \draw[ultra thick, dashed] (1,1) -- (2,1);
        \draw[ultra thick, dashed] (1,1) -- (2,0);
        \node[shape=circle, fill=red, inner sep=3pt,label={left :$u_{a_k}$}] at (0,0) {}; 
    \end{tikzpicture}
    \caption{Pattern observed when $a_{k+1} - a_k = 3$.}
    \label{fig:CN3}
\end{figure}

We assume without loss of generality that $u_{a_k} = (0,0)$, which implies that $u_{a_k+2} = (1,1)$, $A_{k+1} = X_{u_{a_{k+1}}}-2$ and $A_k = Y_{0,0}$. Since $u$ is a path of $\Gamma_{p,R}^{\cn}(P)$, we have 
$$\|P_{1,1} - P_{0,0}\|_p^p = (X_{1,1}-X_{0,0})^p + (Y_{1,1}-Y_{0,0})^p \leq (2R)^p.$$

Using $X_{0,0}\in[0,1]$ and $Y_{1,1}\in[1,2]$, we obtain
$$(X_{1,1}-1)^p + (1-Y_{0,0})^p \leq (2R)^p.$$

It follows that
\begin{equation}\label{eq:StairsIneq}
    X_{1,1} \leq 1 + \big[(2R)^p - (1-Y_{0,0})^p\big]^{\frac{1}{p}}. 
\end{equation}

Since $X_{u_{a_{k+1}}}\leq X_{1,1} + R$, it implies that
\begin{align*}
    A_{k+1}-A_k & \leq X_{1,1}+R-2-Y_{0,0} \\
    & \leq \big[\big(2R\big)^p - (1-Y_{0,0})^p\big]^{\frac{1}{p}} +R-1-Y_{0,0} \\ 
    & < \big[\big(2f(p)\big)^p - (1-Y_{0,0})^p\big]^{\frac{1}{p}} +f(p)-1-Y_{0,0} \text{ (because $R< f(p)$).}
\end{align*}

Observe that $Y_{0,0}\in[0,R]$ because $s_{a_k} \in \{\nwa,\ua\}$. To conclude, we show that for all $x\in[0,R]$, the following inequality holds:
$$(2f(p))^p\leq (1+x-f(p))^p + (1-x)^p.$$

To this purpose, we introduce the function $g(x) = (1+x-f(p))^p + (1-x)^p$, for $x\in[0,R]$. By studying its derivative, we obtain that $g$ admits a minimum in $f(p)/2$. Thus, for all $x\in[0,R]$,
$$g(x)\geq g\bigg(\frac{f(p)}{2}\bigg) =2\bigg(1-\frac{f(p)}{2}\bigg)^p=(2f(p))^p.$$

As a consequence, we obtain that, for any $Y_{0,0} \in [0,R]$,
$$ \big[\big(2R\big)^p - (1-Y_{0,0})^p\big]^{\frac{1}{p}} +R-1-Y_{0,0} < 0.$$
Since $[0,R]$ is compact, there exists $c>0$ such that $A_{k+1}-A_k \leq -c.$

\paragraph{Conclusion:}
For any $k\geq 2$, we have $A_{k+1}-A_k \leq  -\min\{1-2R,c\} < 0$. 
Since $(A_k)_{k\geq 0}$ is a sequence of positive real number, this leads to a contradiction. 
Consequently, for $R<\min\left\{\frac{1}{2},\frac{2}{1+2^{2-\frac{1}{p}}}\right\}$, the graph $\Gamma_{p,R}^{\cn}(P)$ does not contain any infinite self-avoiding path. 
\subsubsection{Nearest neighbor model and $p\in[1,\infty)$}
Let $p\in\big[1,\infty\big)$, $R< \min\{1/2,R(p)\}$ and let $u=(u_k)_{k\geq0}$ be an infinite self-avoiding path of $\Gamma_{p,R}^{\nn}(P)$ for some $P\in\mathcal{P}$. We now have  $s_k=u_k-u_{k-1}\in\{\la,\da,\ua,\ra\}$ for $k\geq1$.

\paragraph{Staircase pattern.}
We introduce the set $\Stairs=\{n\geq1 : s_n=s_{n+2} \mbox{ and } s_{n+1} = s_{n+3}\}$, which contains positions at which the path presents a staircase pattern of length at least 4. We first prove the following.

\begin{lemma} \label{lem:StairsFinite}
    The set $\Stairs$ is infinite.
\end{lemma}

\begin{proof}
    Let us assume by contradiction that $\Stairs$ is finite, and let $M=\max\Stairs$, with the convention that $M=0$ if $\Stairs=\emptyset$. Even if it means replacing $u$ by $(u_k)_{k\geq M}$, we can assume without loss of generality that $\Stairs=\emptyset$.
    Up to a rotation and a symmetry, we may also assume that $(s_1,s_2)=\left(\ra,\ua\right)$, by Lemma~\ref{lem:R<p/q}(i). We have then $s_3\in\{\la,\ra\}$. 
    \begin{itemize}
        \item  If $s_3=\la$, the self-avoiding property and Lemma~\ref{lem:R<p/q}(i) imply that $s_4=\ua$. Then, since $\Stairs=\emptyset$, we have $s_5=\ra$, and by induction, we obtain that $s_{2k+2}=\ua$ and $(s_{4k+1},s_{4k+3})=(\ra,\la)$ for all $k\geq 0$. Let us set $D_k = Y_{u_{2k}} - y_{u_{2k}}$, for $k\geq 0$. This quantity represents the distance of the points $P_{u_{2k}} = (X_{u_{2k}},Y_{u_{2k}})$ to the bottom border of their cells. We have $Y_{u_{2k}+2}\leq Y_{u_{2k}} +2R$ and $y_{u_{2k}+2} = y_{u_{2k}} +1$, so that $D_{k+1}-D_k \leq -(1-2R)<0$. Since $D_k>0$ for all $k\geq 0$, this leads to a contradiction. 
        \item  If $s_3=\ra$, we obtain analogously that $s_{2k+1}=\ra$ and $(s_{4k+2},s_{4k+4})=(\ua,\da)$ for all $k\geq 0$, and we conclude in the same way with $D_k = X_{u_{2k}} - x_{u_{2k}}$. \qedhere
	\end{itemize}
\end{proof}

\paragraph{Notations.}
The rest of the proof follows the same approach as in Section~\ref{sec:LBR_MinCN}. 
Using the fact that $\Stairs$ is infinite, we define a sequence $(a_k)_{k\geq1}$, with the help of two additional sequences $(n_k)_{k\geq1}$ and $(b_k)_{k\geq1}$.

The sequence $(n_k)_{k\geq1}$ is defined by
\begin{itemize}
	\item $n_1 = \min \{n\geq 1 : n\in \Stairs\}$,
	\item for $k\geq1$, 
	$n_{k+1} = \min \{n\geq n_k+2 : n\in \Stairs\}$,
\end{itemize}
and the sequence $(b_k)_{k\geq1}$ by 
\begin{itemize}
	\item $b_1 = 1$,
	\item for $k\geq1$, 
	\begin{equation}\label{eq:b_k} 
	b_{k+1} = \left\{
	\begin{array}{ll}
		b_k +1 &\mbox { if } n_{b_k+1} - n_{b_k} \neq 3,\\
		b_k +2 &\mbox { if } n_{b_k+1} - n_{b_k} = 3 \mbox{ and } n_{b_k+2} - n_{b_k+1} \neq 3,\\
		b_k +3 &\mbox { if } n_{b_k+1} - n_{b_k} = 3 \mbox{ and } n_{b_k+2} - n_{b_k+1} = 3.\\
	\end{array}
	\right.
	\end{equation}
\end{itemize}

For $k \geq 1$, we set $a_k=n_{b_k}$, $\delta_k = s_{a_k}$, and we then define $A_k$ as in Equation~\eqref{eq:Ak}. 

\paragraph{List of patterns between $u_{n_k}$ and $u_{n_{k+1}}$.}

Like in Section~\ref{sec:LBR_MinCN}, we first describe the patterns observed between the points $u_{n_k}$ and $u_{n_{k+1}}$, and then give an upper bound for $A_{k+1}-A_{k}$ in all the possible cases.

\begin{lemma}\label{lem:Description}
    Let $k\geq1$ be such that $s_{n_k}= s_{n_k+2}= \ra$ and $s_{n_k+1}= s_{n_k+3} = \ua$. Then, 
\begin{itemize}
    \item if $n_{k+1} - n_k = 2l$ with $l\geq1$, then between the points $u_{n_{k}+1}$ and $u_{n_{k+1}+1}$, the path alternates between the patterns  $(\ra,\ua)$ and $(\ra,\da)$, i.e.,
    $$\forall m\in\{2,\ldots, 2l+1\}, \quad s_{n_k + m} =
    \begin{cases}
        \ra & \text{if } m = 0 \mod 2,\\
        \da & \text{if } m = 1 \mod 4,\\
        \ua & \text{if } m =3 \mod 4.
    \end{cases}
    $$
    \item if $n_{k+1} - n_k = 2l+1$ with $l\geq1$, then between the points $u_{n_{k}+1}$ and $u_{n_{k+1}+2}$, the paths alternates between the patterns $(\ra,\ua)$ and $(\la,\ua)$, i.e., 
    $$\forall m\in\{2,\ldots,2l+2\}, \quad s_{n_k + m} = 
    \begin{cases}
        \ua & \text{if } m  = 1 \mod 2,\\
        \la & \text{if } m  = 0 \mod 4,\\
        \ra & \text{if } m  = 2 \mod 4.
    \end{cases}
    $$
\end{itemize}
	In particular, for any $k \geq 1$,
	$$s_{n_{k+1}} = \left\{
	\begin{array}{ll}
		\ra &\mbox{ if } n_{k+1} - n_k \mbox{ is even},\\
		\ua &\mbox{ if } n_{k+1} - n_k \mbox{ is odd}.  
	\end{array}
	\right.$$
\end{lemma}

\begin{proof}
We use the same ideas as in the proof of Lemma~\ref{lem:StairsFinite}. 
\end{proof}

Figure~\ref{fig:Description} represents the patterns between $u_{n_k}$ and $u_{n_{k+1}}$ for $n_{k+1} - n_k$ even (green path) and odd (blue path).

\begin{figure}
    \centering
    \begin{tikzpicture}
        \draw [ultra thick] (0,0) -- (1,0) -- (1,1) -- (2,1) -- (2,2);
        \draw [dashed,ultra thick,color=violet] (2,2) -- (1,2) -- (1,3) -- (2,3) -- (2,4) -- (1,4) -- (1,5);
        \draw [dashed,ultra thick,color=darkpastelgreen] (2,2) -- (3,2) -- (3,1) -- (4,1) -- (4,2) -- (5,2) -- (5,1);
        \node[fill=none,draw=none] at (3.5,0.5) {\color{darkpastelgreen}$n_{k+1}-n_k=2l$};
        \node[fill=none,draw=none] at (-.8,3) {\color{violet}$n_{k+1}-n_k=2l+1$};
        \node[inner sep=3pt,fill=blue,shape=circle,label={right :$u_{n_k}$}] at (1,0) {};
        
        \node[fill=none] at (5.5,1.5) {$\cdots$};
        \node[fill=none] at (1.5,5.5) {$\vdots$};
        
        \node[inner sep=2pt,fill=darkpastelgreen,shape=circle,] at (2,1) {};
        \node[inner sep=2pt,fill=darkpastelgreen,shape=circle,] at (3,2) {};
        \node[inner sep=2pt,fill=darkpastelgreen,shape=circle,] at (4,1) {};
        \node[inner sep=2pt,fill=darkpastelgreen,shape=circle,] at (5,2) {};

        \node[inner sep=2pt,fill=violet,shape=circle,] at (2,2) {};
        \node[inner sep=2pt,fill=violet,shape=circle,] at (1,3) {};
        \node[inner sep=2pt,fill=violet,shape=circle,] at (2,4) {};
        \node[inner sep=2pt,fill=violet,shape=circle,] at (1,5) {};

    \end{tikzpicture}
    \caption{Types of pattern observed between $u_{n_{k}}$ and $u_{n_{k+1}}$, depending on the parity of $n_{k+1}-n_k$. The purple and green dots represents the possible positions of $u_{n_{k+1}}$.}
    \label{fig:Description}
\end{figure}
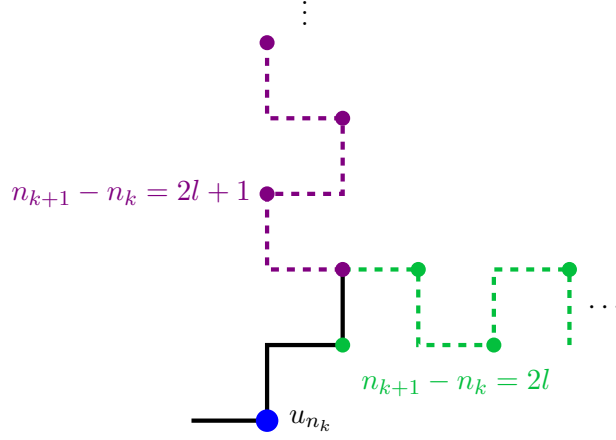

\paragraph{Two important lemmas.}
For the rest of the proof, we use the two following lemmas.

\begin{lemma}\label{lem:Step5}
	Let $p\in[1,\infty)$ and $R< \min\{1/2,R(p)\}$. We consider the sequence $u=(u_k)_{0\leq k\leq6}$ defined by
	$$\begin{array}{llll}
	u_0 = (0,0), & u_1 = (1,0), & u_2 = (1,1), & u_3 = (2,1),\\
	u_4 = (2,2), & u_5 = (1,2), & u_6 = (1,3),&
	\end{array}
	$$
	see Figure~\ref{fig:Step5}. 
	Then, there exists $c_1>0$ such that for any $P\in\mathcal P$ such that $u$ is a path of $\Gamma_{p,R}^{\nn}(P)$, we have 
	\begin{equation}\label{eq:odd}
	(Y_{1,3} - 3) - (X_{1,0} - 1) \leq -c_1 <0.
	\end{equation}
\end{lemma}
    
See Appendix~\ref{App:Lem1} for the proof.

\begin{figure}[ht!]
    \centering
    \begin{tikzpicture}
	    \draw[ultra thick] (0,0) -- (1,0) -- (1,1) -- (2,1) -- (2,2) -- (1,2) -- (1,3);
	    \node[fill=black,inner sep=3pt,shape=circle,label={left :$(0,0)$}] at (0,0) {};
	    \node[fill=red,inner sep=3pt,shape=circle,label={right :$(1,0)$}] at (1,0) {};
	    \node[fill=black,inner sep=3pt,shape=circle,label={right :$(2,2)$}] at (2,2) {};
	    \node[fill=black,inner sep=3pt,shape=circle,label={left :$(1,3)$}] at (1,3) {};
    \end{tikzpicture}
    \caption{The path $u$ of Lemma~\ref{lem:Step5}. The red point corresponds to a position $u_{a_k}$.}
    \label{fig:Step5}
\end{figure}
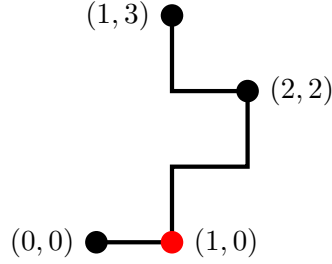
    
\begin{lemma}\label{lem:Step3+odd}
    Let $p\in[1,\infty)$, $R<\min\{1/2,R(p)\}$. We consider the sequence $u=(u_k)_{0\leq k\leq9}$ defined by
	$$\begin{array}{llll}
	u_0 = (0,0), & u_1 = (1,0), & u_2 = (1,1), & u_3 = (2,1),\\
	u_4 = (2,2), & u_5 = (1,2), & u_6 = (1,3), & u_7 = (0,3),\\
	u_8 = (0,2), & u_9 = (-1,2), &&
	\end{array}
	$$
	see Figure~\ref{fig:Step3+odd}. 
	Then, there exists $c_2>0$ such that for any $P\in\mathcal{P}$ such that $u$ is a path of $\Gamma_{p,R}^{\nn}(P)$, we have 
	\begin{equation}\label{eqodd}
	(1 + (-1) - X_{-1,2}) - (X_{1,0}-1) \leq -c_2<0.
	\end{equation}
\end{lemma}

See Appendix~\ref{App:Lem2} for the proof.

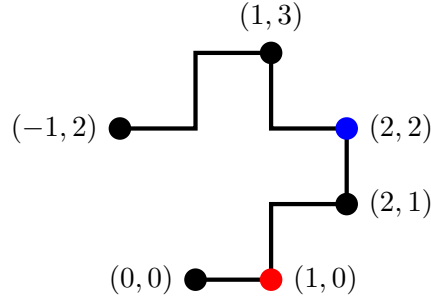
\begin{figure}[ht!]
    \centering
    \begin{tikzpicture}
	    \draw[ultra thick] (0,0) -- (1,0) -- (1,1) -- (2,1) -- (2,2) -- (1,2) -- (1,3) -- (0,3) -- (0,2) -- (-1,2);
	    \node[fill=black,inner sep=3pt,shape=circle,label={left :$(0,0)$}] at (0,0) {};
	    \node[fill=red,inner sep=3pt,shape=circle,label={right :$(1,0)$}] at (1,0) {};
	    \node[fill=black,inner sep=3pt,shape=circle,label={right :$(2,1)$}] at (2,1) {};
	    \node[fill=blue,inner sep=3pt,shape=circle,label={right :$(2,2)$}] at (2,2) {};
	    \node[fill=black,inner sep=3pt,shape=circle,label={above :$(1,3)$}] at (1,3) {};
	    \node[fill=black,inner sep=3pt,shape=circle,label={left :$(-1,2)$}] at (-1,2) {};
    \end{tikzpicture}
    \caption{The path $u$ of Lemma~\ref{lem:Step3+odd}. The red and the blue point correspond respectively to a position $u_{a_k} = u_{n_{b_k}}$ and $ u_{n_{{b_k}+1}}$.}
    \label{fig:Step3+odd}
\end{figure}
    
Lemmas~\ref{lem:Description},~\ref{lem:Step5} and~\ref{lem:Step3+odd} are important and they permit, in the following, to upperbound $A_{k+1}-A_k$.\par

\bigskip
    
Let $k\geq1$ be fixed. We suppose without loss of generality that $s_{a_k}=s_{a_k+2}=\ra$ and $s_{a_k+1}= s_{a_k+3} = \ua$.

\paragraph{Case \boldmath$b_{k+1}-b_k = 1$.} By Equation~\eqref{eq:b_k}, $n_{b_{k}+1} - n_{b_k}\not=3$. We now distinguish the even case and the odd case.

\begin{itemize}
    \item {\bf Even case:} $n_{b_{k+1}} - n_{b_k} = 2l$ with $l\geq1$. \\
    Then point $u_{a_{k+1}}$ is a green point on Figure~\ref{fig:Description}.
    By Lemma~\ref{lem:Description}, we have $\delta_{k+1} = \delta_k = \ra$ and $x_{u_{a_{k+1}}}= x_{u_{a_{k}}} + l$. So, $X_{u_{a_k+2l}} \leq X_{u_{a_k}} + 2lR$. We obtain then
    \begin{align*}
        A_{k+1} - A_{k} &= (X_{u_{a_{k+1}}} - x_{u_{a_{k+1}}}) - (X_{u_{a_k}} - x_{u_{a_k}})\\
        &\leq-2l\left(\frac{1}{2}-R\right) \leq -(1-2R) < 0.
    \end{align*}
    \item {\bf Odd case:} $n_{b_{k+1}} - n_{b_k} = 2l+1$ with $l\geq2$.\\
    The point $u_{a_{k+1}}$ is a purple point, except the first one (because $l \geq 2$), on Figure~\ref{fig:Description}. To upperbound $A_{k+1}-A_k$, we split this value into two terms. One is upperbound by Lemma~\ref{lem:Step5} and the other one using Lemma~\ref{lem:Description}.
    We have
    \begin{align*}
        A_{k+1} - A_{k} &= (Y_{u_{a_{k+1}}} - y_{u_{a_{k+1}}}) - (X_{u_{a_k}} - x_{u_{a_k}})\\
        &= (Y_{u_{a_{k}+2l+1}} -y_{u_{a_{k}+2l+1}}) - (Y_{u_{a_{k}+5}} - y_{u_{a_{k}+5}}) \\
        & \quad + (Y_{u_{a_{k}+5}} - y_{u_{a_{k}+5}}) - (X_{u_{a_k}} - x_{u_{a_k}}).\\
    \end{align*}
    
     Up to a translation, $(u_l)_{a_k-1\leq l \leq a_{k+5}}$ satisfies the hypotheses of Lemma~\ref{lem:Step5}. It follows that 
    $$(Y_{u_{a_{k}+5}} - y_{u_{a_{k}+5}}) - (X_{u_{a_k}} - x_{u_{a_k}}) \leq -c_1.$$
    
    By applying Lemma~\ref{lem:Description} to $(u_l)_{l\geq a_{k+2}}$, we have $y_{u_{a_{k}+2l+1}} = y_{u_{a_{k}+5}} + l-2$. Since $u$ is a path of $\Gamma_{p,R}^{\nn}(P)$, it implies that
    $$Y_{u_{a_{k}+2l+1}} \leq  Y_{u_{a_{k}+5}} + 2(l-2)R.$$
    
    So, we obtain
    $$A_{k+1} - A_k \leq -2(l-2)\left(\frac{1}{2}-R\right) -c_1 \leq -c_1 < 0.$$
\end{itemize}

\paragraph{Case \boldmath$b_{k+1}-b_k = 2$.} By definition (see Equation~\eqref{eq:b_k}), we have $n_{b_k+1} - n_{b_k} = 3$ and 
either $n_{b_k+2} - n_{b_k+1} = 2l$ with $l\geq1$, or $n_{b_{k}+2} - n_{b_k+1} = 2l+1$ with $l\geq2$. Such configurations are represented on Figure~\ref{fig:diffb=2}.

\begin{figure}
    \centering
    \begin{tikzpicture}
        \draw [ultra thick] (0,0) -- (1,0) -- (1,1) -- (2,1) -- (2,2) -- (1,2) -- (1,3) -- (0,3);
        \draw [dashed,ultra thick,color=darkpastelgreen] (0,3) -- (0,4) -- (1,4) -- (1,5) -- (0,5);
        \draw [dashed,ultra thick,color=violet] (0,3) -- (0,2) -- (-1,2) -- (-1,3) -- (-2,3);
        \node[fill=none,draw=none] at (-1,1.5) {\color{violet}$n_{b_k+2}-n_{b_k+1}=2l+1$};
        \node[fill=none,draw=none] at (3,4) {\color{darkpastelgreen}$n_{b_k+2}-n_{b_k+1}=2l$};
        \node[inner sep=3pt,fill=red,shape=circle,label={right :$u_{a_k}$},label={below :{\color{red}$\ra$}}] at (1,0) {};
        \node[inner sep=3pt,fill=blue,shape=circle,label={right :$u_{n_{b_k+1}}$}] at (2,2) {};
        
        \node[inner sep=2pt,fill=darkpastelgreen,shape=circle,] at (1,3) {};
        \node[inner sep=2pt,fill=darkpastelgreen,shape=circle,] at (0,4) {};
        \node[inner sep=2pt,fill=darkpastelgreen,shape=circle,] at (1,5) {};

        \node[inner sep=2pt,fill=violet,shape=circle,] at (-1,2) {};
        \node[inner sep=2pt,fill=violet,shape=circle,] at (-2,3) {};

    \end{tikzpicture}
    \caption{Pattern observed between the points $u_{n_{b_k}}$ and $u_{n_{b_{k+1}}}$ when $b_{k+1}-b_k = 2$. The red arrow corresponds to $\delta_k$. The purple and green dots represents the possible positions of $u_{n_{b_{k+1}}} =u_{n_{b_{k}+2}} $.
    We remark that, starting from the point $u_{n_{b_k+1}-1}$, the pattern is similar to the one observed in Figure~\ref{fig:Description}.}
    \label{fig:diffb=2}
\end{figure}
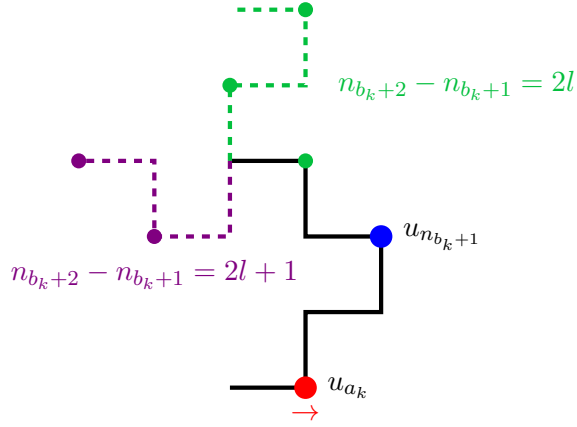

\begin{itemize}
    \item {\bf Even case:} $n_{b_k+2} - n_{b_k+1} = 2l$ with $l\geq1$.\\
    The point $u_{a_{k+1}}$ is a green point on Figure~\ref{fig:diffb=2}. 
    As in the previous case ($b_{k+1}-b_k = 1$, odd case), we obtain
    $$ A_{k+1} - A_k \leq -2(l-1)\left(\frac{1}{2}-R\right)-c_1 \leq -c_1.$$ 
    
    \item{\bf Odd case:} $n_{b_k+2} - n_{b_k+1} = 2l+1$ with $l\geq2$. \\
    The point $u_{a_{k+1}}$ is a purple point on Figure~\ref{fig:diffb=2}.
    By Lemma~\ref{lem:Description}, we have $A_{k+1} = 1+ x_{u_{a_{k+1}}}-X_{u_{a_{k+1}}}$. As previously, to upperbound $A_{k+1}-A_k$, we split it into two terms. One is upperbounded by Lemma~\ref{lem:Step3+odd} and the other one using Lemma~\ref{lem:Description}.
    
    \begin{align*}
        A_{k+1} - A_k &= (1+x_{u_{a_{k+1}}}-X_{u_{a_{k+1}}} ) - (X_{{u_{a_k}}}-x_{{u_{a_k}}})\\
        & = (1+x_{u_{a_{k+1}}}-X_{u_{a_{k+1}}}) - (1+ x_{u_{a_k+8}}-X_{u_{a_k+8}}) \\
        & \quad + (1+x_{u_{a_k+8}}-X_{u_{a_k+8}}) - (X_{{u_{a_k}}}-x_{{u_{a_k}}}).
    \end{align*}
    
     Up to a translation, $(u_l)_{a_k-1\leq l \leq a_{k+8}}$ satisfies the hypotheses of Lemma~\ref{lem:Step3+odd}. It follows that
    $$(1+ x_{u_{a_k+8}} -X_{u_{a_k+8}}) - (X_{{u_{a_k}}} - x_{{u_{a_k}}})\leq -c_2.$$
    
    By Lemma~\ref{lem:Description}, we have  $x_{u_{a_{k+1}}} = x_{u_{a_k+8}} - (l-2)$ which implies that
    $$(1+x_{u_{a_{k+1}}}-X_{u_{a_{k+1}}}) - (1+ x_{u_{a_k+8}}-X_{u_{a_k+8}}) \leq -2(l-2)\left(\frac{1}{2}-R\right).$$
   
    So, we obtain
    $$A_{k+1} - A_k \leq -2(l-2)\left(\frac{1}{2}-R\right) - c_2 \leq - c_2.$$
\end{itemize}

\paragraph{Case \boldmath$b_{k+1}-b_k = 3$.} By Equation~\eqref{eq:b_k}, we have  $n_{b_k+1} - n_{b_k} = n_{b_k+2} - n_{b_k+1} = 3$. Such configurations are represented on Figure~\ref{fig:diffb=3}.

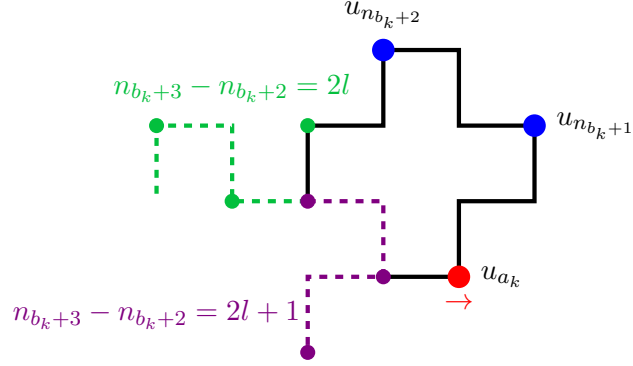
\begin{figure}
    \centering
    \begin{tikzpicture}
        \draw [ultra thick] (0,0) -- (1,0) -- (1,1) -- (2,1) -- (2,2) -- (1,2) -- (1,3) -- (0,3) -- (0,2) -- (-1,2) -- (-1,1);
        \draw [dashed,ultra thick,color=violet] (-1,1) -- (0,1) -- (0,0) -- (-1,0) -- (-1,-1);
        \draw [dashed,ultra thick,color=darkpastelgreen] (-1,1) -- (-2,1) -- (-2,2) -- (-3,2) -- (-3,1);
        \node[fill=none,draw=none] at (-3,-.5) {\color{violet}$n_{b_k+3}-n_{b_k+2}=2l+1$};
        \node[fill=none,draw=none] at (-2,2.5) {\color{darkpastelgreen}$n_{b_k+3}-n_{b_k+2}=2l$};
        \node[inner sep=3pt,fill=red,shape=circle,label={right :$u_{a_k}$},label={below :{\color{red}$\ra$}}] at (1,0) {};
        \node[inner sep=3pt,fill=blue,shape=circle,label={right :$u_{n_{b_k+1}}$}] at (2,2) {};
        \node [fill=blue,shape=circle,inner sep =3pt,shape=circle,label={above :$u_{n_{b_k+2}}$}] at (0,3) {};
        
        \node[inner sep=2pt,fill=darkpastelgreen,shape=circle,] at (-1,2) {};
        \node[inner sep=2pt,fill=darkpastelgreen,shape=circle,] at (-2,1) {};
        \node[inner sep=2pt,fill=darkpastelgreen,shape=circle,] at (-3,2) {};
        
        \node[inner sep=2pt,fill=violet,shape=circle,] at (-1,1) {};
        \node[inner sep=2pt,fill=violet,shape=circle,] at (0,0) {};
        \node[inner sep=2pt,fill=violet,shape=circle,] at (-1,-1) {};
    \end{tikzpicture}
    \caption{Pattern observed between the points $u_{n_{b_{k}}}$ and $u_{n_{b_{k+1}}}$ when $b_{k+1}-b_k = 3$. The purple and green dots represents the possible positions of $u_{n_{b_{k+1}}}$.
    We remark that, starting from the point $u_{n_{b_k+2}-1}$, the pattern is similar to the one observed in Figure~\ref{fig:Description}.}
    \label{fig:diffb=3}
\end{figure}

\begin{itemize}
    \item {\bf Even case:} $n_{b_k+3} - n_{b_k+2} = 2l$ with $l\geq1$.\\
    The point $u_{a_{k+1}}$ is a green point on Figure~\ref{fig:diffb=3}. 
    As in the previous case ($b_{k+1}-b_k = 2$, odd case), we obtain
    $$ A_{k+1} - A_k \leq -2(l-1)\left(\frac{1}{2}-R\right)-c_2 \leq -c_2.$$
    
    \item {\bf Odd case:} Suppose that $n_{b_k+3} - n_{b_k+2} = 2l+1$ with $l\geq1$. By Lemma~\ref{lem:Description}, we notice that $u_{n_{b_k}-1} = u_{n_{b_k}+11}$, see Figure~\ref{fig:diffb=3}. So $u$ is not a self-avoiding path and this case is excluded.
\end{itemize}

\paragraph{Conclusion.} For any $k\geq 1$, we proved that $A_{k+1} - A_k\leq -\min\{1-2R,c_1,c_2\}<0$. This concludes the proof.

\section{Critical radius ($R_c$)}\label{sec:Rc}
\subsection{Proof of Theorem \ref{thm:Rc}}\label{sec:proofRc}

Let $p\in[1,\infty]$.

\paragraph{Upper Bound.}

Since $R_c^{\cn}(p)\leq R_c^{\nn}(p)$, we only have to prove the upper bound for $R_{c}^{\nn}(p)$. To this end, we use a coupling argument, and compare  $\Gamma_{p,R}^{\nn}(P)$ with a Bernoulli percolation configuration $G_B(P)$ associated to the same set $P$ of points.

Let $\eps>0$, and let $R=\|(1,1+\eps)\|_p$. We define the set of edges  $E_B(P)$ of the (undirected) graph $G_B(P) = (\ZZ^2,E_B(P))$ by, for $i,j\in\ZZ$:
\begin{itemize}
    \item $\big((i,j),(i+1,j)\big)\in E_B(P) \iff X_{i,j}>i+1-\eps$,
    \item $\big((i,j),(i,j+1)\big)\in E_B(P) \iff Y_{i,j}>j+1-\eps$,
\end{itemize}
see an example of construction in Figure~\ref{fig:UBR_C}.

We clearly have $E_B(P)\subset E_{p,R}^{\nn}(P)$. Hence, if $G_B(P)$ percolates (i.e.\ has an infinite connected component), then $\Gamma_{p,R}^{\nn}(P)$ percolates.

Moreover, under the law $\mathbb{P}$ on $P$ defined in the Introduction, observe that each edge of the square lattice belongs to $E_B(P)$ with probability $\eps$, independently of the other edges.
Since the threshold for Bernoulli bond percolation on $\ZZ^2$ is equal to $1/2$, it follows that, for any $R > \|(1,\frac{3}{2})\|_p$, $\mathbb{P}(|\C_B(P)|=\infty)>0$.

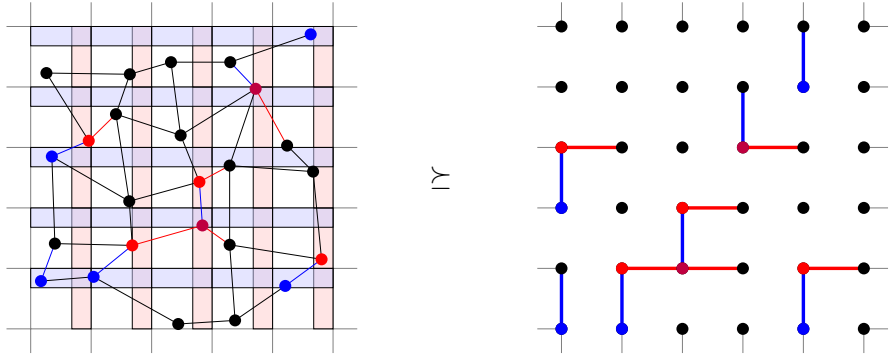
\begin{figure}
    \centering
    \begin{minipage}{.35\textwidth}
    \scalebox{.8}{%
    \begin{tikzpicture}
    \draw[step=1cm,color=gray,ultra thin] (-2.4,-2.4) grid (3.4,3.4);
	\foreach \i in {-2,...,2}{
	    \foreach \j in {-2,...,2}{
    		\draw[draw=black,fill=red!20,fill opacity=0.5] (\i+.68,\j) rectangle ++(0.32,1);
    		\draw[draw=black,fill=blue!20,fill opacity=0.5] (\i,\j+.68) rectangle ++(1,0.32);
	}}

    \draw (-1.83,-1.21) -- (-0.96,-1.14);
    \draw [color=blue] (-1.83,-1.21) -- (-1.6,-0.59);
    \draw (-1.6,-0.59) -- (-0.32,-0.62);
    \draw (-1.6,-0.59) -- (-1.65,0.85);
    \draw (-1.65,0.85) -- (-0.37,0.11);
    \draw [color=blue] (-1.65,0.85) -- (-1.04,1.11);
    \draw [color=red] (-1.04,1.11) -- (-0.59,1.55);
    \draw (-1.04,1.11) -- (-1.74,2.23);
    \draw (-1.74,2.23) -- (-0.36,2.21);
    \draw (-0.96,-1.14) -- (0.44,-1.92);
    \draw [color=blue] (-0.96,-1.14) -- (-0.32,-0.62);
    \draw [color=red] (-0.32,-0.62) -- (0.84,-0.29);
    \draw (-0.32,-0.62) -- (-0.37,0.11);
    \draw (-0.37,0.11) -- (0.79,0.43);
    \draw (-0.37,0.11) -- (-0.59,1.55);
    \draw (-0.59,1.55) -- (0.48,1.2);
    \draw (-0.59,1.55) -- (-0.36,2.21);
    \draw (-0.36,2.21) -- (0.32,2.41);
    \draw (0.44,-1.92) -- (1.38,-1.86);
    \draw [color=red] (0.84,-0.29) -- (1.29,-0.61);
    \draw [color=blue] (0.84,-0.29) -- (0.79,0.43);
    \draw [color=red](0.79,0.43) -- (1.29,0.7);
    \draw (0.79,0.43) -- (0.48,1.2);
    \draw (0.48,1.2) -- (1.72,1.97);
    \draw (0.48,1.2) -- (0.32,2.41);
    \draw (0.32,2.41) -- (1.3,2.41);
    \draw (1.38,-1.86) -- (2.21,-1.29);
    \draw (1.38,-1.86) -- (1.29,-0.61);
    \draw (1.29,-0.61) -- (2.81,-0.85);
    \draw (1.29,-0.61) -- (1.29,0.7);
    \draw (1.29,0.7) -- (2.67,0.6);
    \draw (1.29,0.7) -- (1.72,1.97);
    \draw [color=red] (1.72,1.97) -- (2.24,1.03);
    \draw [color=blue] (1.72,1.97) -- (1.3,2.41);
    \draw (1.3,2.41) -- (2.63,2.87);
    \draw [color=blue] (2.21,-1.29) -- (2.81,-0.85);
    \draw (2.81,-0.85) -- (2.67,0.6);
    \draw (2.67,0.6) -- (2.24,1.03);
    
    \node [fill=blue,inner sep=2pt,shape=circle] at (-1.83,-1.21) {};
    \node [fill=black,inner sep=2pt,shape=circle] at (-1.6,-0.59) {};
    \node [fill=blue,inner sep=2pt,shape=circle] at (-1.65,0.85) {};
    \node [fill=red,inner sep=2pt,shape=circle] at (-1.04,1.11) {};
    \node [fill=black,inner sep=2pt,shape=circle] at (-1.74,2.23) {};
    \node [fill=blue,inner sep=2pt,shape=circle] at (-0.96,-1.14) {};
    \node [fill=red,inner sep=2pt,shape=circle] at (-0.32,-0.62) {};
    \node [fill=black,inner sep=2pt,shape=circle] at (-0.37,0.11) {};
    \node [fill=black,inner sep=2pt,shape=circle] at (-0.59,1.55) {};
    \node [fill=black,inner sep=2pt,shape=circle] at (-0.36,2.21) {};
    \node [fill=black,inner sep=2pt,shape=circle] at (0.44,-1.92) {};
    \node [fill=purple,inner sep=2pt,shape=circle] at (0.84,-0.29) {};
    \node [fill=red,inner sep=2pt,shape=circle] at (0.79,0.43) {};
    \node [fill=black,inner sep=2pt,shape=circle] at (0.48,1.2) {};
    \node [fill=black,inner sep=2pt,shape=circle] at (0.32,2.41) {};
    \node [fill=black,inner sep=2pt,shape=circle] at (1.38,-1.86) {};
    \node [fill=black,inner sep=2pt,shape=circle] at (1.29,-0.61) {};
    \node [fill=black,inner sep=2pt,shape=circle] at (1.29,0.7) {};
    \node [fill=purple,inner sep=2pt,shape=circle] at (1.72,1.97) {};
    \node [fill=black,inner sep=2pt,shape=circle] at (1.3,2.41) {};
    \node [fill=blue,inner sep=2pt,shape=circle] at (2.21,-1.29) {};
    \node [fill=red,inner sep=2pt,shape=circle] at (2.81,-0.85) {};
    \node [fill=black,inner sep=2pt,shape=circle] at (2.67,0.6) {};
    \node [fill=black,inner sep=2pt,shape=circle] at (2.24,1.03) {};
    \node [fill=blue,inner sep=2pt,shape=circle] at (2.63,2.87) {};
\end{tikzpicture}
    }
    \end{minipage}
    \hspace{-.2cm}$\succeq$\hspace{1cm}
    \begin{minipage}{.35\textwidth}
    \scalebox{.8}{%
    \begin{tikzpicture}
        \foreach \i in {0,...,5}{
            \draw[color=gray,ultra thin] (-0.4,\i) -- (0,\i);
            \draw[color=gray,ultra thin] (5.4,\i) -- (5,\i);
            \draw[color=gray,ultra thin] (\i,5) -- (\i,5.4);
            \draw[color=gray,ultra thin] (\i,0) -- (\i,-.4);
        }

	    \draw [color=blue,ultra thick] (0,0) -- (0,1);
	    \draw [color=blue,ultra thick] (4,4) -- (4,5);
	    \draw [color=blue,ultra thick] (0,2) -- (0,3);
	    \draw [color=blue,ultra thick] (1,0) -- (1,1);
	    \draw [color=red,ultra thick] (0,3) -- (1,3);
	    \draw [color=red,ultra thick] (1,1) -- (2,1);
	    \draw [color=red,ultra thick] (2,2) -- (3,2);
	    \draw [color=red,ultra thick] (4,1) -- (5,1);
	    \draw [color=blue,ultra thick] (2,1) -- (2,2);
	    \draw [color=red,ultra thick] (2,1) -- (3,1);
	    \draw [color=blue,ultra thick] (3,3) -- (3,4);
	    \draw [color=red,ultra thick] (3,3) -- (4,3);
	    \draw [color=blue,ultra thick] (4,0) -- (4,1);
	    \foreach \i in {0,...,5}{
	        \foreach \j in {0,...,5}{
        		\node[inner sep=2pt,shape=circle,fill=black] at (\i,\j) {};
	    }}
	    \node[inner sep=2pt,shape=circle,fill=blue] at (0,0) {};
        \node[inner sep=2pt,shape=circle,fill=blue] at (4,4) {};
	    \node[inner sep=2pt,shape=circle,fill=blue] at (1,0) {};
	    \node[inner sep=2pt,shape=circle,fill=blue] at (0,2) {};
        \node[inner sep=2pt,shape=circle,fill=red] at (0,3) {};
        \node[inner sep=2pt,shape=circle,fill=red] at (1,1) {};
        \node[inner sep=2pt,shape=circle,fill=red] at (2,2) {};
        \node[inner sep=2pt,shape=circle,fill=red] at (4,1) {};
        \node[inner sep=2pt,shape=circle,fill=purple] at (2,1) {};
        \node[inner sep=2pt,shape=circle,fill=purple] at (3,3) {};
        \node[inner sep=2pt,shape=circle,fill=blue] at (4,0) {};
    \end{tikzpicture}
    }
    \end{minipage}
	\caption{On the left, the nearest neighbor configuration $\Gamma_{p,R}^{\nn}(P)$, and on the right the Bernoulli percolation configuration $G_B(P)$. We place a horizontal (resp. vertical) edge in $G_B(P)$ if the corresponding point of $P$ lies inside the red (resp. blue) rectangle.}
	\label{fig:UBR_C}
\end{figure}

\paragraph{Lower bound.}

We again use a coupling argument. Let $\Gamma_{R}^{\dec}(P)= (\ZZ^2,E_{R}^{\dec}(P))$ be the (undirected) subgraph of the square lattice whose set of edges $E_{R}^{\dec}(P)$ satisfies, for $i,j\in\ZZ$:
\begin{itemize}
    \item $\big((i,j),(i+1,j)\big)\in E_{R}^{\dec}(P) \iff  |X_{i,j} - X_{i+1,j}|\leq R$,
    \item $\big((i,j),(i,j+1)\big)\in E_{R}^{\dec}(P) \iff |Y_{i,j} - Y_{i,j+1}|\leq R$.
\end{itemize}

We say that such a graph $\Gamma_R^{\dec}(P)$ is a configuration of the \emph{decorrelated model} of parameter $R$, and we denote by $R_c^{\dec}$ the critical radius of this new model. 
For any $P\in\mathcal{P}$, we have $E_{p,R}^{\nn}(P)\subset E_R^{\dec}(P)$. Thus, 
$R_c^{\dec}\leq R_c^{\nn}(p).$
Therefore, it suffices to establish the lower bound
$$\sqrt{2-\sqrt{2}}\leq R_c^{\dec}.$$

To do that, we show that if $R < \sqrt{2-\sqrt{2}}$, then the dual graph of $\Gamma_R^{\dec}(P)$ almost surely percolates.
We denote this dual graph by $\overline\Gamma_R^{\dec}(P) = \left(\left(\frac{1}{2},\frac{1}{2}\right)+\ZZ^2,\overline E_R^{\dec}(P)\right)$. 
Recall that, by definition, $e\in E_R^{\dec}(P) \iff \overline e \notin \overline E_R^{\dec}(P)$, where for an edge $e$ of the standard square lattice, we denote by $\overline{e}$ the corresponding (orthogonal) edge in the dual lattice.

Let us consider the set $E'$ of edges of the dual square lattice obtained by conserving only the edges of one line over two, and of one column over two, see Figure~\ref{fig:LBR_cDec}. Mathematically, it is defined by:
\begin{align*}
    E'= & \left\{\left((1/2+2i,1/2+j),(1/2+2i,3/2+j)\right) : i,j \in \ZZ\right\} \\
    &  \quad \cup \left\{\left((1/2+i,1/2+2j),(3/2+i,1/2+2j)\right) : i,j \in \ZZ \right\}.
\end{align*}

We claim that each edge of $E'$ belongs to $\overline E_R^{\dec}(P)$ with probability
$$1-p_R = 1-\mathbb{P}\left(\left((0,0),(1,0)\right)\in E_R^{dec}(P)\right)=1-\mathbb{P}\left(|X_{0,0} - X_{1,0}|\leq R\right),$$
independently of the other edges of $E'$.

Indeed, the horizontal edge $((1/2+2i,1/2+j),(1/2+2i,3/2+j)) \in E'$ belongs to $\overline E_R^{\dec}(P)$ if and only if $|Y_{2i,j} - Y_{2i,j+1}| > R$. This happens with probability $1-p_R$, and furthermore, two distinct horizontal edges depend on different random variables. So, all horizontal edges are independent. Similarly, all vertical edges are independent. And, since horizontal edges depend on $Y$ random variables, and vertical edges on $X$ random variables, all edges are independent.

By grouping consecutive edges of $E'$ in pairs, and comparing with Bernoulli bond percolation on the square lattice of parameter $(1-p_R)^2$, we obtain that if $(1-p_R)^2>1/2$, then the restriction of $\overline\Gamma_R^{\dec}(P)$ to the edges of $E'$ has almost surely an infinite connected component, so that $\overline\Gamma_R^{\dec}(P)$ percolates almost surely. Furthermore, standard results on Bernoulli percolation imply that in this case, every vertex of $\ZZ^2$ is almost surely surrounded by the infinite connected component, so that $\Gamma_R^{\dec}(P)$ has almost surely no infinite connected component.
For $R\leq1$, we have $p_R=\frac{R^2}{2}$. As a consequence, if $R<\sqrt{2-\sqrt{2}}$, then $\Gamma_R^{\dec}(P)$ has almost surely no infinite connected component. It follows that $R_c^{\dec}\geq \sqrt{2-\sqrt{2}}$.

\begin{figure}[ht!]
\centering
	\begin{tikzpicture}
		\begin{scope}[xshift=0.5 cm,yshift=0.5 cm]
				\foreach \i in {-4,...,3}{
				\foreach \j in {-4,...,3}{
					\node[inner sep=2pt,fill=black,shape=circle] at (\i,\j) {};
			}}
			\draw[step=2cm,color=black,ultra thick,dashed] (-4.8,-4.8) grid (3.8,3.8);
			\begin{scope}[xshift=0.5cm,yshift=0.5cm]
			    \draw[step=1cm,color=gray,ultra thin] (-5.3,-5.3) grid (3.3,3.3);
			\end{scope}
			\foreach \i in {-2,...,1}{
			\foreach \j in {-2,...,1}{
				\node[inner sep=2pt,fill=black] at (2*\i,2*\j) {};
			}}
			\begin{scope}[yshift=1 cm]
			\foreach \i in {-2,...,1}{
			\foreach \j in {-2,...,1}{
				\node[inner sep=2pt,fill=black] at (2*\i,2*\j) {};
			}}
			\end{scope}
			\begin{scope}[xshift=1 cm]
			\foreach \i in {-2,...,1}{
			\foreach \j in {-2,...,1}{
				\node[inner sep=2pt,fill=black] at (2*\i,2*\j) {};
			}}
			\end{scope}
		\end{scope}
	\end{tikzpicture}
	\caption{The underlying grey grid is the square lattice $\ZZ^2$, and the dashed edges are the edges belonging to the set $E'$. The points are the vertices of the dual lattice $(1/2,1/2)+\ZZ^2$.}
	\label{fig:LBR_cDec}
\end{figure}
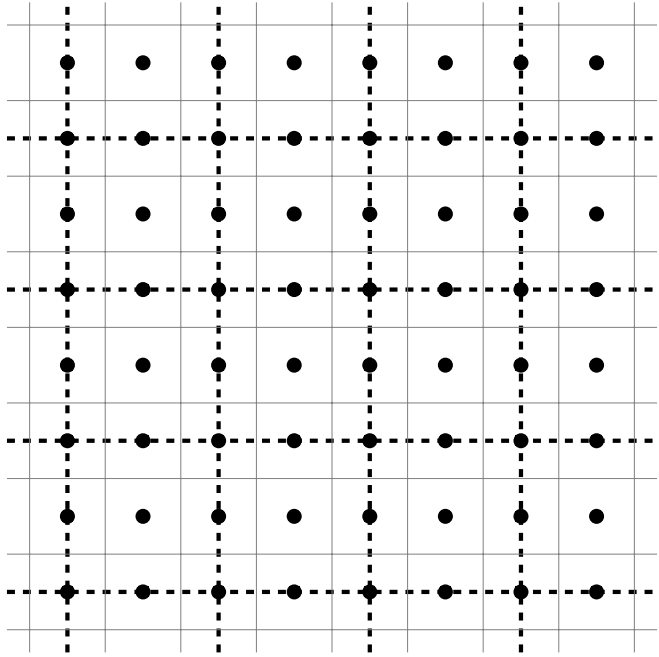

\subsection{Estimates}\label{sec:Estimates}

In Figure~\ref{fig:R_cEsti}, we provide estimates of the critical radius for some values of the parameter $p$, both for the complete neighbor model and for the nearest neighbor one.

\begin{figure}[ht!]
    \centering
    \begin{tabular}{|c|c|c|c|c|c|c|c|c|}
        \hline
        $p$ & 1 & 1.5 & 2 & 2.5 & 3 & 3.5 & 4 & $\infty$ \\
        \hline
        $R_c^{\cn}$ & 1.305 & 1.120 & 1.040 & 1.005 & 0.985 & 0.980 & 0.965 & 0.880 \\
        \hline
        $R_c^{\nn}$ & 1.350 & 1.120 & 1.095 & 1.050 & 1.040 & 1.030 & 1.025 & 1 \\
        \hline
    \end{tabular}
    \caption{Estimates of the critical radius for both the complete and the nearest neighbor model.}
    \label{fig:R_cEsti}
\end{figure}

To obtain these values, we estimated the probability for the origin to be connected to a border of the grid $[-N,N]^2$, for $N=40$. Precisely, for each $R=0.005\times n$ with $n\geq 0$, we ran $1000$ simulations of the models on the grid $[-40,40]^2$. The approximations of $R_c$ given in the table then correspond to the first integers $n\geq 0$ such that for at least one of the $1000$ simulations, the origin was connected to a border of the grid.

In Figure~\ref{fig:ResultsCN} and \ref{fig:ResultsNN}, our estimates of the critical are represented with crosses.

\section{Perspectives}
\paragraph{Critical radius.}

Theorem~\ref{thm:Rc} provides lower and upper bounds of the critical radius, for both the complete neighbor and the nearest neighbor model. However, the exact values remain to be determined. In Section~\ref{sec:Estimates}, we propose numerical estimations of the critical radius, obtained with the help of computer simulations. They suggest that for the nearest neighbor model  with $p=\infty$, the critical radius is equal to $1$. For $p=\infty$ and $R\geq 1$, $\Gamma_{\infty,1}^{\nn}(P)$ coincides with the decorrelated model $\Gamma_1^{\dec}(P)$ introduced in Section~\ref{sec:Rc}. We thus state the following conjecture.

\begin{conjecture}
$R_c^{\nn}(\infty) = R_c^{\dec} = 1$.
\end{conjecture}

 This conjecture is supported by the following properties.
\begin{itemize}
    \item For $R=1$, the density of edges is equal to:
    $${\mathbb P}(((0,0),(0,1))\in E_{\infty,1}^{\nn}(P))={\mathbb P}(X_{1,0}-X_{0,0}\leq 1)=1/2,$$
    a value that coincides with the critical probability for Bernoulli percolation on the square lattice. Moreover, the states of any two edges are independent, unless they constitute consecutive edges on a same horizontal or vertical line. Observe also that the probability of any finite pattern has a rational value, that can be computed in terms of counting permutations. 
    \item The decorrelated model exhibits a remarkable symmetry property with respect to the value  $R=1$. Precisely, for $R\in[0,2]$, let us denote by $\widetilde{\Gamma}_{R}^{\dec}(P)$ the subgraph of the lattice obtained from $\Gamma_{R}^{dec}(P)$ by switching the open and closed edges. Then, $\widetilde{\Gamma}_R^{\dec}(P)$ and $\Gamma_{2-R}^{\dec}(P)$ have the same distribution.
\end{itemize}

\paragraph{Unicity of the infinite connected component.}

In the super-critical regime, we conjecture that there is almost surely a unique infinite connected components, both for the complete model and for the nearest neighbor one. It seems that the proof strategy developed by Burton and Keane~\cite{BK89} for Bernoulli percolation can be used to prove the uniqueness for the nearest neighbor model with $p=\infty$ (at least), but what about the general case?

\paragraph{Asymptotic shape of the connected component of the origin.} For a graph $G=(\ZZ^2, E)$, let us denote by ${\mathcal B}_n(G)$ the set of points of $\ZZ^2$ that are at distance smaller than $n$ from the origin $(0,0)$ for the graph distance in $G$, that is: 
$${\mathcal B}_n(G)=\{v\in \ZZ^2 : \exists u_0=(0,0),u_1,\ldots, u_{n-1},u_n=v \in \ZZ^2, \forall  i\in\{0,\ldots, n-1\}, (u_i,u_{i+1})\in E\}.$$

We expect the diameter of ${\mathcal B}_n(\Gamma_{p,R}^{\cn}(P))$ to grow linearly, at least for $R\geq R^{\cn}_{max} (p)$. 
What can we say about the asymptotic shape of the set $$\frac{1}{n}\Big({\mathcal B}_n(\Gamma_{p,R}^{\cn}(P))+[0,1]^2\Big)$$
when $n$ tends to the infinity? When $R$ goes to infinity, does this shape become close to the ball of the corresponding ${\mathcal L}_p$-norm? And what about the counterpart of this shape for the nearest neighbor model?

\paragraph{Gilbert's disc model conditioned on other lattices.}
In this article, we focused on Gilbert's disc model conditioned on the square lattice. The definitions naturally extend to other lattices (triangular or hexagonal lattice, $d$-dimensional lattices...). What are then the values of the possible connectivity radius, of the critical radius, and of the total connectivity radius?

\bibliographystyle{alpha}
\newcommand{\etalchar}[1]{$^{#1}$}

\section{Appendix}
\subsection{Proof of Lemma~\ref{lem:Step5}}\label{App:Lem1}

    Let $u$ be a path of $\Gamma_{p,R}^{\nn}(P)$.
    Firstly, the pattern observed between the points $u_0$ and $u_4$ is one of the patterns observed for the case $a_{k+1}-a_k=3$ at the end of section~\ref{sec:LBR_MinCN}.
    So by similar arguments used for Equation~\eqref{eq:StairsIneq}, we have
    \begin{equation}\label{eq:Y_(2,1)}
        Y_{2,1} \leq 1+ \left[(2R)^p - \left(1 - (X_{1,0}-1) \right)^p\right]^{\frac{1}{p}}.  
    \end{equation}
    
    Since $Y_{1,3}\leq Y_{2,1} + 3R$, we have
    \begin{equation}\label{eq:R<R(p)L1.1}
        (Y_{1,3} - 3) - (X_{1,0} - 1) \leq \big[(2R)^p - \left(1 - (X_{1,0}-1) \right)^p\big]^{\frac{1}{p}} + 3R - 2 - (X_{1,0}-1).
    \end{equation}

    \paragraph{Case \boldmath$p\leq \displaystyle\frac{\ln(2)}{\ln(\frac{4}{3})}$.}
    By Lemma~\ref{lem:R(p)}, $\min\{1/2,R(p)\} = 1/2$. Using the inequality $R<1/2$, we obtain
    \begin{equation}\label{eq:R<R(p)L1.2}
        \big[(2R)^p - \left(1 - (X_{1,0}-1) \right)^p\big]^{\frac{1}{p}} + 3R - 2 - (X_{1,0}-1) < \big[1 - \left(1 - (X_{1,0}-1) \right)^p\big]^{\frac{1}{p}} - \frac{1}{2} - (X_{1,0}-1).
    \end{equation}
    To conclude, we show that for all $x\in[0,R]$, the following inequality holds:
    $$1 \leq (1-x)^p + \left(\frac{1}{2}+x \right)^p.$$
    To this purpose, we set the function $g(x)=(1-x)^p + \left(\frac{1}{2}+x\right)^p$, for $x\in[0,R]$.
    
    By studying its derivative, we obtain that $g$ admits a minimum in $x= 1/4$. Thus, for all $x\in[0,R]$,
    $$g(x)\geq g\left(\frac{1}{4}\right) = 2\left(\frac{3}{4}\right)^p\geq 2\left(\frac{3}{4}\right)^{\frac{\ln(2)}{\ln(\frac{4}{3})}}=1.$$
    
    As a consequence, taking $x = X_{1,0}-1$, we obtain that, for any  $x\in[0,R]$, 
    $$\big[1 - \left(1 - (X_{1,0}-1) \right)^p\big]^{\frac{1}{p}} - \frac{1}{2} - (X_{1,0}-1) \leq 0.$$
    And, by Equation~\eqref{eq:R<R(p)L1.2} and since $[1,1+R]$ is compact, there exists $c_1 > 0$,
    $$ \big[\big(2R\big)^p - (1-(X_{1,0}-1))^p\big]^{\frac{1}{p}} +3R -2 -(X_{1,0}-1) \leq -c_1<0,$$
    and so, by Equation~\eqref{eq:R<R(p)L1.1}, $(Y_{1,3} - 3) - (X_{1,0} - 1) \leq -c_1 < 0.$

    \paragraph{Case \boldmath $p> \displaystyle\frac{\ln(2)}{\ln(\frac{4}{3})}$.}
    
    By Lemma~\ref{lem:R(p)}, $\min\{1/2,R(p)\} = R(p)$, hence $R< R(p)$, and so
    \begin{align}
    &  \big[(2R)^p - \left(1 - (X_{1,0}-1) \right)^p\big]^{\frac{1}{p}} + 3R - 2 - (X_{1,0}-1) \nonumber\\
    &\quad < \big[(2R(p))^p - \left(1 - (X_{1,0}-1) \right)^p\big]^{\frac{1}{p}} + 3R(p) - 2 - (X_{1,0}-1) \label{eq:R<R(p)L1.3} .
    \end{align}
    To conclude, we show that for all $x\in[0,R]$, the following inequality holds:
    $$(2R(p))^p \leq (1-x)^p + (2+x-3R(p))^p.$$
    To this purpose, we set the function $g(x)=(1-x)^p+(2+x-3R(p))^p$, for $x\in[0,R]$.
    
    By studying its derivative, we obtain that $g$ admits a minimum in $(3R(p)-1)/2$. Thus, for all $x\in[0,R]$,
    $$g(x)\geq g\left(\frac{3R(p)-1}{2}\right) = 2\left(\frac{3-3R(p)}{2}\right)^p.$$
    As $R(p) < 1/2$ and by its definition, see Lemma~\ref{lem:R(p)},
    $$2\left(\frac{3-3R(p)}{2}\right)^p\geq \bigg(\frac{3-3R(p)}{2}\bigg)^p + \bigg(\frac{2-R(p)}{2}\bigg)^p=(2R(p))^p.$$
    
    As a consequence, taking $x = X_{1,0}-1$, we obtain that, for any $x\in[0,R]$,
    $$ \big[\big(2R(p)\big)^p - (1-(X_{1,0}-1))^p\big]^{\frac{1}{p}} +3R(p) -2 -(X_{1,0}-1) \leq0.$$
    And, by Equation~\eqref{eq:R<R(p)L1.3} and since $[1,1+R]$ is compact, there exists $c_1 > 0$,
    $$ \big[\big(2R\big)^p - (1-(X_{1,0}-1))^p\big]^{\frac{1}{p}} +3R -2 -(X_{1,0}-1) \leq -c_1<0,$$
    and so, by Equation~\eqref{eq:R<R(p)L1.1}, $(Y_{1,3} - 3) - (X_{1,0} - 1) \leq -c_1 < 0.$

\subsection{Proof of  Lemma~\ref{lem:Step3+odd}}\label{App:Lem2}

    Let $u$ be a path of $\Gamma_{p,R}^{\nn}(P)$. Then we have 
    $$\|P_{1,3}-P_{2,2}\|_p^p = (X_{1,3}-X_{2,2})^p + (Y_{1,3}-Y_{2,2})^p \leq (2R)^p.$$
    
    Using $Y_{1,3}\in[3,4]$ and $X_{2,2}\in[2,3]$, we obtain  
    $$(2-X_{1,3})^p +(3-Y_{2,2})^p\leq (2R)^p.$$
    
    Since $Y_{2,2} < Y_{2,1} + R(p)$ and by Equation~\eqref{eq:Y_(2,1)}, it follows that
    \begin{align*}
        X_{1,3}&\geq 2 - \big((2R)^p-(3-Y_{2,2}))^p\big)^{\frac{1}{p}}\\
        &\geq 2 - \bigg[(2R)^p-\bigg(2-R-\big[(2R)^p-(1-(X_{1,0}-1))^p\big]^{\frac{1}{p}}\bigg)^p\bigg]^{\frac{1}{p}}.
	\end{align*}
	Since $X_{-1,2} \geq X_{1,3} - 3$, it implies that
	\begin{align}
        (1 + (-1) - X_{-1,2}) - (X_{1,0}-1) &\leq  3R- X_{1,3} - (X_{1,0}-1)\\
        & \leq \left[(2R)^p-\left(2-R-\left[(2R)^p-(1-(X_{1,0}-1))^p\right]^{\frac{1}{p}}\right)^p\right]^{\frac{1}{p}}\nonumber \\
        &+ 3R - 2 - (X_{1,0}-1)\label{eq:R<R(p)L2.1} \\
        & < \left[(2R(p))^p-\left(2-R(p)-\left[(2R(p))^p-(1-(X_{1,0}-1))^p\right]^{\frac{1}{p}}\right)^p\right]^{\frac{1}{p}} \nonumber \\
        &+ 3R(p) - 2 - (X_{1,0}-1)\label{eq:R<R(p)L2.2}
    \end{align}
    
    To conclude, we show that for all $x\in[0,R]$, the following inequality hold:
    $$[(2R(p))^p - (2+x-3R(p))^p]^{\frac{1}{p}} + [(2R(p))^p - (1-x)^p]^{\frac{1}{p}}< 2-R(p).$$

    To this purpose, we set the function 
    $$g(x) =  [(2R(p))^p - (1 - x)^p]^{\frac{1}{p}} + [(2R(p))^p - (2 + x - 3R(p))^p]^{\frac{1}{p}} ,$$
    for $x\in[0,R]$. By studying its derivative, we obtain that $g$ admits a maximum in $(3R(p)-1)/2$. Thus, for all $x\in[0,R]$,
    $$g(x)\leq g\left(\frac{3R(p)-1}{2}\right) = 2\left[(2R(p))^p - \left(\frac{3-3R(p)}{2}\right)^p\right]^{\frac{1}{p}}.$$
    
    By Equation~\eqref{eqRmin4V}, we obtain 
    $$ g\left(\frac{3R(p)-1}{2}\right) = 2-R(p).$$
    
    As a consequence, taking $x=X_{1,0}-1$, we obtain that, for any $X_{1,0}\in[1,1+R]$,
    $$\bigg[(2R(p))^p-\bigg(2-R(p)-\big[(2R(p))^p-(1-(X_{1,0}-1))^p\big]^{\frac{1}{p}}\bigg)^p\bigg]^{\frac{1}{p}}+ 3R(p) - 2 - (X_{1,0}-1)\leq 0.$$
    And, by Equation~\eqref{eq:R<R(p)L2.2} and since $[1,1+R]$ is compact, there exists $c_2 > 0$ such that
    $$\bigg[(2R)^p-\bigg(2-R-\big[(2R)^p-(1-(X_{1,0}-1))^p\big]^{\frac{1}{p}}\bigg)^p\bigg]^{\frac{1}{p}}+ 3R - 2 - (X_{1,0}-1)\leq -c_2,$$ 
    and so, by Equation~\eqref{eq:R<R(p)L2.1}, $(1 + (-1) - X_{-1,2}) - (X_{1,0} - 1) \leq -c_2$.

\end{document}